\documentclass[11pt,leqno]{article}

\usepackage{amssymb,amsmath,amsthm}

\newcommand{\n}{\noindent}
\newcommand{\bb}[1]{\mathbb{#1}}
\newcommand{\cl}[1]{\mathcal{#1}}
\newcommand{\vp}{\varepsilon}

\newcommand{\sst}{\scriptstyle}
\newcommand{\intl}{\int\limits}
\newcommand{\ms}{\medskip}

\theoremstyle{plain}
\newtheorem{thm}{Theorem}[section]
\newtheorem{pro}[thm]{Proposition}
\newtheorem{lem}[thm]{Lemma}
\newtheorem{cor}[thm]{Corollary}
\newtheorem*{prb}{Problem}

\theoremstyle{remark}
\newtheorem*{rk}{Remark}
\newtheorem{rem}[thm]{Remark}

\theoremstyle{definition}
\newtheorem{defn}[thm]{Definition}
\newtheorem{stp}{Step}

\numberwithin{equation}{section}

\begin{document}

\title{Remarks on the non-commutative Khintchine inequalities for $0<p<2$}

\author{by\\
Gilles Pisier\footnote{Partially supported by NSF grant 0503688
and   ANR-06-BLAN-0015.}\\
Texas A\&M University\\
College Station, TX 77843, U. S. A.\\
and\\
 Universit\'e Paris 6 (UPMC)\\
Institut Math.   Jussieu (Analyse Fonctionnelle)\\ Case 186, 75252
Paris Cedex 05, France}

\date{Nov.  14, 2008}

\maketitle

\begin{abstract} 
We show that the validity of the non-commutative Khintchine inequality for some $q$ with $1<q<2$ implies its validity (with another constant) for all $1\le p<q$. We prove this
for the inequality involving the Rademacher functions, but also for more general ``lacunary'' sequences, or even non-commutative analogues of the Rademacher functions. For instance, we may apply it to the ``$Z(2)$-sequences'' previously considered by Harcharras. The result appears to be new in that case. It implies that the space $\ell^n_1$ contains (as an operator space) a large subspace uniformly isomorphic (as an operator space) to $R_k+C_k$ with $k\sim n^{\frac12}$. This naturally raises several interesting questions concerning the best possible such $k$.
Unfortunately we cannot settle the validity of the non-commutative Khintchine inequality
for $0<p<1$ but we can prove several would be corollaries. For instance,   given an infinite scalar matrix $[x_{ij}]$, we give a necessary and sufficient condition for $[\pm  x_{ij}]$ to be in the Schatten class $S_p$ for almost all (independent) choices of signs $\pm 1$.
We also   characterize the bounded Schur multipliers from $S_2$ to $S_p$.
The latter two characterizations extend to $0<p<1$ results already known for $1\le p\le2$.
In  addition, we observe that the hypercontractive inequalities, proved by Carlen and Lieb
for the Fermionic case, remain valid for operator space valued functions, and hence
the Kahane inequalities are valid in this setting.
 
 \end{abstract}

2000 MSC 46L51, 46L07, 47L25, 47L20
\vfill\eject

The non-commutative Khintchine inequalities play a very important r\^ole in the recent developments
in non-commutative Functional Analysis, and in particular in Operator Space Theory, see \cite{P2,P4}. Just like their commutative counterpart for ordinary $L_p$-spaces, they are    a central tool to understand
all  sorts of questions involving series of random
variables, or random vectors, in relation with unconditional or almost unconditional convergence
in non-commutative $L_p$ (\cite{PX}). The commutative version is  also crucial 
in the factorization theory for linear maps between $L_p$-spaces
\cite{Ma,Ma2}   in connection with Grothendieck's Theorem.
The non-commutative analogues 
of Grothendieck's Theorem reflect the same close connection with the Khintchine inequalities,
see e.g. the recent paper \cite{X}.
Moreover, in the non-commutative case, 
further motivation for their study comes from Random Matrix Theory and
Free Probability. For instance one finds that the Rademacher functions
(i.e. i.i.d. $\pm 1$-valued) independent random  variables satisfy the same inequalities
as the freely independent ones in non-commutative $L_p$ for $p<\infty$.

For reasons that hopefully will appear below, the case $p<2$ is more delicate, and 
actually the case $p<1$ is still open. When $p<2$,  let us  say for convenience  that a sequence $(f_k)$
in classical $L_p$ satisfies the classical Khintchine inequality $KI_p$ if there is a constant $c_p$
such that for all finite scalar sequences $(a_j)$ we have
$$(\sum |a_j|^2)^{1/2}\le c_p \| \sum a_j f_j \|_p.$$
Now assume that $(f_k)$ is orthonormal in $L_2$.
Then it is easy to see that if $p<q<2$,  $KI_q$ implies $KI_p$. Indeed,
let $S=\sum a_j f_j$.  Let $\theta$ be such that
$ 1/q= (1-\theta)/p+\theta/2.$  We have 
$$(\sum |a_j|^2)^{1/2}\le c_q\|S\|_q \le c_q\|S\|^{1-\theta}_p \|S\|^{\theta}_2=c_q\|S\|^{1-\theta}_p (\sum |a_j|^2)^{\theta/2}, $$
and hence after a suitable division we obtain $KI_p$
with $c_p=(c_q)^{1/(1-\theta)}$. The heart of this simple argument is
   that the span of the sequence $(f_k)$ is the same in $L_p$  and in $L_q$
   or in $L_2$.  
In sharp contrast, the analogue of this fails for operator spaces. The span of the
 Rademacher functions in $L_p$ is not isomorphic as operator space
 to its span in $L_q$, although they have  the same underlying Banach space. This is reflected
 in the form of the non-commutative version of the Khintchine inequalities
 first proved by Lust-Piquard in \cite{LP1} and labelled as $(Kh_q)$ below
 for the case of  non-commutative $L_q$. Nevertheless, it turns out that
 the above simple minded extrapolation argument can still be made to work,
 this is our main result
 but this requires a more sophisticated version of H\"older's inequality,
 that (apparently) forces us to restrict  ourselves to $p\ge 1$.
 
Let $1\le q\le 2$. Let $(r_k)$ be the Rademacher functions on $\Omega = [0,1]$. Let $(x_k)$ be a finite sequence in a non-commutative $L_q$-space. The non-commutative Khintchine inequalities say (when $1\le q\le 2$) that there is a constant $\beta_q$ independent of $x=(x_k)$ such that 
\begin{equation}\label{eq0.1}
|||x|||_q \le \beta_q\left(\int\left\|\sum r_k(t) x_k\right\|^q_q dt\right)^{1/q}
\end{equation}
where
\begin{equation}\label{eq0.2}
|||x|||_q \overset{\sst\text{def}}{=} \inf_{x_k=a_k+b_k} \left\{\left\|\left(\sum a^*_ka_k\right)^{1/2} \right\|_q + \left\|\left(\sum b_kb^*_k\right)^{1/2}\right\|_q\right\}.
\end{equation}
This was first proved in \cite{LP1} for $1<q<2$ and in \cite{LPP} for $q=1$ (the converse inequality is easy and holds
with constant $1$). One of the two proofs in \cite{LPP} derives this from the non-commutative (little) Grothendieck inequality proved in \cite{P1}.

In this paper, we follow an approach very similar to the original one in \cite{P1} to show that the validity of \eqref{eq0.1} for some $q$ with $1<q<2$ implies its validity (with another constant) for all $1\le p<q$;
we also make crucial use of more recent ideas from \cite{JP}. For that deduction the only assumption needed on $(r_k)$ is its orthonormality in $L_2([0,1])$. Thus our approach yields \eqref{eq0.1} also for more general ``lacunary'' sequences than the Rademacher functions. For instance, we may apply it to the ``$Z(2)$-sequences'' considered in \cite{Har} (see also \cite{Har2,BHar}). The result appears to be new in that case. Our argument can be viewed  as an operator space analogue of the classical fact, in Rudin's style (\cite{Ru}), that if a sequence of characters $\Lambda$  spans a Hilbert space in $L_q(G)$ ($G$ compact Abelian group, e.g.\ $G={\bb T}$) for some $q<2$ then it also does for all $p<q$. 
It implies that the space $\ell^n_1$ contains (as an operator space) a large subspace uniformly isomorphic (as an operator space) to $R_k+C_k$ with $k\sim n^{\frac12}$.
Another corollary (see Theorem \ref{sigmap}) is that 
there is a constant $c$ such that, for any $n$, the usual ``basis" of $S_1^n$ contains  
a $c$-unconditional subset of size $\ge   n^{3/2}$. This opens the door to 
various questions concerning the best possible size of such subspaces
and subsets. See the end of \S \ref{sec1} for some speculation on this.

Unfortunately we cannot prove our result (at the time of this writing) for $0<p<1$, for lack of a proof of Step 3 below. Thus we leave open the validity of \eqref{eq0.1} for $0<q<1$. Nevertheless we will be able to prove several partial results in that direction. In particular  (see \S \ref{sec3} and \S \ref{sec4}),
if $0<p\le 2$, given 
arbitrary scalar
coefficients $[x_{ij}]$, 
 we give a necessary and sufficient condition for the random matrix  
$$ [\pm x_{ij}]$$
to be in the Schatten class $S_p$ for almost all choices of signs.
This happens iff $[x_{ij}]$ admits a decomposition of the form
  $x_{ij} = a_{ij}+b_{ij}$ with
\[
\sum_i\left(\sum_j|a_{ij}|^2\right)^{p/2} <\infty \quad \text{and}\quad \sum_j\left(\sum_i |b_{ij}|^2\right)^{p/2}<\infty.
\]
We also show
that $[x_{ij}]$ defines a bounded Schur multiplier from $S_2$ to $S_p$ iff
it admits a decomposition of the form 
$
   x_{ij}= \psi_{ij}+\chi_{ij} $  {with}
$$\sum\limits_i \sup\limits_j|\psi_{ij}|^{2p/(2-p)}<\infty\quad{\rm and}\quad \sum\limits_j \sup\limits_i|\chi_{ij}|^{2p/(2-p)}<\infty.$$
In those two results, only the case $0<p<1$ is new. In passing we remind the reader that
when $0<p<1$, $L_p$-spaces (commutative or not), in particular the Schatten class $S_p$,  are
not normed spaces. They are only $p$-normed, i.e. for any pair $x,y$ in the space
we have $\|x+y\|^p\le \|x\|^p+ \|y\|^p$. 

In  the final section, we turn to the Kahane inequalities.
Recall  that  the latter are a vector valued version of the Khintchine inequalities valid for functions with values in an \emph{arbitrary} Banach space. It is natural to wonder whether
there are non-commutative analogues when one uses
the  \emph{operator space valued} non-commutative $L_p$-spaces introduced in \cite{P4}.
We observe that the hypercontractive inequalities, proved by Carlen and Lieb \cite{CL}
for the Fermionic case, remain valid for operator space valued functions,
and hence the Kahane inequalities
are valid in this Fermionic setting.
The point of this   simple remark  is that Kahane's inequality now appears as the Bosonic
case. The same remark is valid
    in any setting for which hypercontractivity has been
established. This applies  in particular to Biane's free hypercontractive inequalities \cite{Bi}.

\section{The case $\pmb{1\le p<2}$}\label{sec1}

Actually, $L_2([0,1])$ or $L_2(G)$ can be replaced here by any non-commutative $L_2$-space $L_2(\varphi)$ associated to a semi-finite generalized (i.e.\ ``non-commutative'') measure space,
and $(r_k)$ is then replaced by an
orthonormal sequence $(\xi_k)$ in $L_2(\varphi)$. Then
     the right-hand side of \eqref{eq0.1} is replaced by $\|\sum \xi_k\otimes x_k\|_{L_q(\varphi\times\tau)}$.
More precisely, by a (semi-finite) generalized measure space $(N,\varphi)$ we mean a von~Neumann algebra $N$ equipped with a faithful, normal, semi-finite trace $\varphi$.  Without loss of generality,
we may always reduce consideration to the $\sigma$-finite case.
Throughout this paper, we will use freely the basics of non-commutative integration as described in \cite{Ne} or \cite[Chap. IX]{Tak2}.

Let us fix another generalized measure space $(M,\tau)$. The inequality we are interested in now takes the following form:
\[
(K_q)\quad \left\{
\begin{array}{l}
\exists \beta_q \text{ such that for any finite sequence}\\
x = (x_k) \text{ in } L_q(\tau) \text{ we have}\\
|||x|||_q \le \beta_q \left\|\sum \xi_k\otimes x_k\right\|_{L_q(\varphi\times\tau)}\\
\text{where } |||\cdot|||_q \text{ is defined as in \eqref{eq0.2}.}
\end{array}\right.
\]

In the Rademacher case, i.e. when $(\xi_k)=(r_k) $,  we denote
$(Kh_q)$ instead of $(K_q)$, and we refer to these as
the non-commutative Khintchine inequalities.

We can now state our main result for the case $q\ge 1$.

\begin{thm}\label{thm1.1}
Let $1<q<2$. Recall that $(\xi_k)$ is assumed orthonormal in $L_2(\varphi)$. \\ Then $(K_q)\Rightarrow (K_p)$ for all $1\le p<q$.
\end{thm}

Here is a sketch of the argument. We denote
\[
 S = \sum \xi_k\otimes x_k.
\]
Let ${\cl D}$ be the collection of all ``densities,'' i.e.\ all $f$ in $L_1(\tau)_+$ with $\tau(f) = 1$. Fix $p$ with $0<p\le q$. Then we denote for $x=(x_k)$ 
\[
C_q(x) = \inf\left\{\left\|\sum \xi_k\otimes y_k\right\|_q\right\}
\]
where $\|\cdot\|_q$ is the norm in $L_q(\varphi\otimes\tau)$ and the infimum runs over all sequences $y = (y_k)$ in $L_q(\tau)$ for which there is $f$ in ${\cl D}$ such that
\[
 x_k = (f^{\frac1p-\frac1q}y_k + y_kf^{\frac1p-\frac1q})/2.
\]
Note that $C_p(x) = \|S\|_p$.

\begin{rem}\label{rem1.2}
 Assume $x_k=x^*_k$ for all $k$. Then
\begin{equation}\label{eq1.1}
 |||x|||_q = \inf\left\{\left\|\left( \sum \alpha^*_k\alpha_k\right)^{1/2}\right\|_q\right\}
\end{equation}
where the infimum runs over all decompositions
\[
 x_k = \text{Re}(\alpha_k) = (\alpha_k+\alpha^*_k)/2.
\]
Indeed, $x_k = a_k+b_k$ implies $x_k = \text{Re}(a_k+b^*_k)$. Let $\alpha_k = a_k+b^*_k$. We have (assuming $q\ge 1$)
\[
\left\|\left(\sum \alpha^*_k\alpha_k\right)^{1/2}\right\|_q \le \left\|\left(\sum a^*_ka_k\right)^{1/2} \right\|_q + \left\|\left(\sum b_kb^*_k\right)^{1/2}\right\|_q.
\]
Therefore $\inf\|(\sum \alpha^*_k\alpha_k)^{1/2}\|_q \le |||x|||_q$. Since the converse inequality is obvious, this proves \eqref{eq1.1}. 
\end{rem}

The proof of Theorem \ref{thm1.1} is based on a variant of ``Maurey's extrapolation principle'' (see \cite{Ma,Ma2}) This combines three steps:\ (here $C',C'',C''',\ldots$ are constants independent of $x=(x_k)$ and we wish to emphasize that here $p$ remains fixed while the index $q$ in $C_q(x)$ is such that $p<q\le2$).\ms

\begin{stp}\label{stp1}
Assuming $(K_q)$ we have
\[
 |||x|||_p \le C'C_q(x).
\]
\end{stp}

\begin{stp}\label{stp2}
\[
C_2(x) \le C''|||x|||_p.
\]
Actually we will prove also the converse inequality (up to a constant).\ms 
\end{stp}

\begin{stp}\label{stp3}
\[
 C_q(x) \le C'''C_p(x)^{1-\theta} C_2(x)^\theta
\]
where $\frac1q = \frac{1-\theta}p + \frac\theta2$. (Recall $p<q<2$ so that $0<\theta<1$.)

The three steps put all together yield
\begin{align*}
 |||x|||_p &\le C'C''' C_p(x)^{1-\theta} (C''|||x|||_p)^\theta\\
\intertext{and hence}
|||x|||_p &\le C^{''''}C_p(x) = C^{''''}\|S\|_p.
\end{align*}
\end{stp}

\begin{proof}[Proof of Step \ref{stp1}]
This is easy:\ We simply apply $(K_q)$ to $y=(y_k)$.  
More precisely, fix $\vp>0$. Let $y=(y_k)$ and $f$ in ${\cl D}$    such that
$x_k = (f^{\frac1p-\frac1q}y_k + y_kf^{\frac1p-\frac1q})/2.$
 and
$$ \left\|\sum \xi_k\otimes y_k\right\|_q< C_q(x) (1+\vp).$$
By $(K_q)$ we have $|||y|||_q< \beta_q C_q(x) (1+\vp).$
Let $a_k,b_k$ be such that $y_k = a_k+b_k$ with
\[
\left\|\left(\sum a^*_ka_k\right)^{1/q}\right\| + \left\|\left(\sum b_kb^*_k\right)^{1/2}\right\|_q \le \beta_q C_q(x) (1+\vp)
\]
we have
\[
 2x_k = f^{\frac1p-\frac1q}(a_k+b_k) + (a_k+b_k)f^{\frac1p-\frac1q}.
\]
But it is easy to check that for some $g,h\in {\cl D}$ there are $\alpha_k,\beta_k$ such that
\[
 a_k = \alpha_k g^{\frac1q-\frac12} \qquad b_k = h^{\frac1q-\frac12}\beta_k
\]
with
\begin{align*}
 \left(\sum\|\alpha_k\|^2_2\right)^{1/2} &\le \left\|\left(\sum a^*_ka_k\right)^{1/2}\right\|_q\\
\intertext{and}
\left(\sum \|\beta_k\|^2_2\right)^{1/2} &\le \left\|\left(\sum b_kb^*_k\right)^{1/2}\right\|_q.
\end{align*}
Thus we find
\[
2x_k = f^{\frac1p-\frac1q} \alpha_k g^{\frac1q-\frac12} + f^{\frac1p-\frac1q} h^{\frac1q-\frac12} \beta_k
+ \alpha_k g^{\frac1q-\frac12} f^{\frac1p-\frac1q} + h^{\frac1q-\frac12} \beta_k f^{\frac1p-\frac1q}.
\]
Let $\frac1r = \frac1p-\frac12$.
Note that by H\"older's inequality (since $\frac1p-\frac12 = (\frac1p-\frac1q) + (\frac1q-\frac12)$)
\[
 \|f^{\frac1p-\frac1q} h^{\frac1q-\frac12}\|_r \le 1\quad {\rm and}\quad  \|g^{\frac1q-\frac12} f^{\frac1p-\frac1q}\|_r \le 1.
\]
Let $x'_k = f^{\frac1p-\frac1q}h^{\frac1q-\frac12}\beta_k + \alpha_k g^{\frac1q-\frac12} f^{\frac1p-\frac1q}$. Then, again by  H\"older, we have
$$
|||x'|||_p  \le \left(\sum \|\alpha_k\|^2_2\right)^{1/2} + \left(\sum\|\beta_k\|^2_2\right)^{1/2}
 \le \beta_qC_q(x) (1+\vp).$$
Similarly,   let 
\[
 x''_k = f^{\frac1p-\frac1q} \alpha_k g^{\frac1q-\frac12} + h^{\frac1q-\frac12} \beta_k f^{\frac1p-\frac1q}.
\]
We claim that
\[
||||x''|||_p \le  \beta_qC_q(x)(1+\vp).
\]
Thus we obtain, since $2x=x'+x''$
\begin{align*}
2|||x|||_p &\le |||x'|||_p + |||x''|||_p \\
&\le 2 \beta_qC_q(x) (1+\vp),
\end{align*}
and hence Step \ref{stp1} holds with $C'= \beta_q$.

We now check the above  claim. Define $\theta$ by $\frac1q = \frac{1-\theta}p + \frac\theta2  $ and let
\[
x_k(z) = f^{\frac1p-\frac1{q(z)}} \alpha_k g^{\frac1{q(z)}-\frac12}
\]
where $\frac1{q(z)} \overset{\sst\text{def}}{=} \frac{1-z}p + \frac{z}2$. 
We will use the 
probability measure  $\mu_\theta$ on the boundary of the complex strip ${\cal S}=\{0<\Re({z})<1\}$ that is
the Jensen (i.e. harmonic)  measure for the point $\theta$. This gives mass $\theta$ (resp. $1-\theta   $) to the vertical line 
 $\{ \Re({z})=1\}$ (resp. $\{ \Re({z})=0\}   $). By perturbation, we may assume
 that $f$ and $g$ are suitably bounded below so that $x_k(.)$ is a ``nice" $L_p(\tau)$-valued 
 analytic  function on $ {\cal S}$, i.e.   bounded and continuous on $\bar{\cal S}$. Then, since $q(\theta)=q$, we have by Cauchy's formula
\[
  f^{\frac1p-\frac1q} \alpha_k g^{\frac1q-\frac12}  = x_k(\theta) = \intl_{\Re(z)\in \{0,1\}} x_k(z)\ d\mu_\theta(z),
\]
 {  but} 
$\forall t\in {\bb R} \quad x_k(it)  = U(it)\alpha_k V(it) g^{\frac1p-\frac12}$  {\rm and }$
x_k(1+it)  =  f^{\frac1p-\frac12} U(1+it)\alpha_k V(1+it)
$,
where $U(it)=U(1+it)=f^{it(\frac1p-\frac12)}$ and $V(it)=V(1+it)=g^{it(\frac12-\frac1p)}$ are unitary. This yields
\[
 f^{\frac1p-\frac1q} \alpha_k g^{\frac1q-\frac12} = (1-\theta)\alpha^{(0)}_k g^{\frac1p-\frac12} + \theta f^{\frac1p-\frac12} \alpha^{(1)}_k
\]
where   $\alpha^{(1)}$  (resp. $\alpha^{(0)}$) are  the corresponding  averages over $\{ \Re({z})=1\}$ (resp. $\{ \Re({z})=0\}$)  satisfying  $(\sum\|\alpha^{(0)}_k\|^2_2)^{1/2} \le (\sum\|\alpha_k||^2_2)^{1/2}$ and  $(\sum\|\alpha^{(1)}_k\|^2_2)^{1/2} \le (\sum\|\alpha_k||^2_2)^{1/2}$, and
hence we find $|||( f^{\frac1p-\frac1q} \alpha_k g^{\frac1q-\frac12})|||_p \le  (\sum\|\alpha_k\|^2_2)^{1/2}$. Similarly, we find $|||( h^{\frac1q-\frac12} \beta_k f^{\frac1p-\frac1q})|||_p \le  (\sum\|\beta_k\|^2_2)^{1/2}$. Thus we obtain as claimed
$$ ||||x''|||_p \le (\sum\|\alpha_k\|^2_2)^{1/2}+(\sum\|\beta_k\|^2_2)^{1/2}\le \beta_qC_q(x)(1+\vp) .$$
\end{proof}
\begin{proof}[Proof of Step \ref{stp2}]
Assume $|||x|||_p<1$, i.e.\ $x_k = a_k+b_k$ with $\|(\sum a^*_ka_k)^{1/2}\|_p + \|(\sum b_kb^*_k)^{1/2}\|_p < 1$. 
By semi-finiteness of $\tau$, we may assume there exists
  $f_0>0$ in ${\cl D}$. (In the finite case
  we can simply take $f_0=1$.) Let $f' = (\vp(f_0)^{2/p} + \sum a^*_ka_k + \sum b_kb^*_k)^{1/2}$. We can choose $\vp>0$ small enough so that
\[
 \|f'\|^p_p < 2 
\]
(using the fact that $L_{p/2}(\tau)$ is $p/2$-normed). We can then write
\[
 x_k = a'_kf' + f'b'_k
\]
where $a'_k = a_k(f')^{-1}$ and $b'_k = (f')^{-1}b_k$. Let $f = (f')^p( \tau(f^{\prime p}))^{-1}$. Note that $f\in{\cl D}$ and we have
\begin{equation}\label{eq1.2}
 x_k = \alpha_k f^{\frac1p-\frac12} + f^{\frac1p-\frac12}\beta_k
\end{equation}
where 
\[
\alpha_k= \|f'\|_p a'_k f^{\frac12}\qquad \beta_k = \|f'\|_p f^{\frac12}b'_k.
\]
Note that
$
 \sum a^{\prime*}_ka'_k  = (f')^{-1} \sum a^*_ka_k(f')^{-1}
\le 1$
 {and similarly}
$\sum b'_kb^{\prime *}_k \le 1$.
Therefore
$$
 \left(\sum\|\alpha_k\|^2_2\right)^{1/2} = \|f'\|_p \left\|\left(\sum a^{\prime *}_ka^{\prime}_k\right)^{1/2} f^{1/2}\right\|_2
 \le \|f'\|_p \le 2^{\frac1p },$$
and similarly
\[
 \left(\sum \|\beta_k\|^2_2\right)^{1/2} \le 2^{\frac1p }.
\]
We will now modify this to obtain
$\alpha_k=\beta_k$. More precisely we claim there are $y_k$ in $L_2(\tau)$ such that 
$x_k = (f^{\frac1p-\frac12}y_k + y_k f^{\frac1p-\frac12})/2$ and 
\begin{equation}\label{eq1.3}
 \left(\sum \|y_k\|^2_2\right)^{1/2} \le 2\left(\left(\sum\|\alpha_k\|^2_2\right)^{1/2} + \left(\sum \|\beta_k\|^2_2\right)^{1/2}\right).
\end{equation}
Let $\frac1r={\frac1p-\frac12}$.
 To prove this claim, let $E$ be the dense subspace
 of $L_2(\tau)\oplus \cdots \oplus L_2(\tau)$ formed
 of families $h=(h_k)$ such that $f^{\frac1r} h_k + h_k f^{\frac1r}\in L_2(\tau)$ for all $k$.
 Then for all $h$ in $E$ we have $\sum \langle x_k,h_k\rangle=
 \sum \langle \alpha_k f^{\frac1r}+f^{\frac1r}\beta_k,h_k\rangle$ and hence
 $$|\sum \langle x_k,h_k\rangle| \le (\sum\|\alpha_k\|^2_2 )^{1/2} (\sum\|h_k f^{\frac1r} \|^2_2 )^{1/2}
 +(\sum\|\beta_k\|^2_2 )^{1/2} (\sum\| f^{\frac1r} h_k \|^2_2 )^{1/2}
 .$$
 By an elementary calculation
 one verifies easily that
 $\|f^{\frac1r} h_k\|^2_2\le \|f^{\frac1r} h_k+h_k f^{\frac1r}\|^2_2$ and similarly
 $\|h_k f^{\frac1r} \|^2_2\le \|f^{\frac1r} h_k+h_k f^{\frac1r}\|^2_2$.
 Therefore we find
$$|\sum \langle x_k,h_k\rangle| \le ((\sum\|\alpha_k\|^2_2 )^{1/2}+(\sum\|\beta_k\|^2_2 )^{1/2})
(\sum\|f^{\frac1r} h_k+h_k f^{\frac1r}\|^2_2)^{1/2}.
$$
From this our claim that there are $(y_k)$ in $L_2(\tau)$ satisfying \eqref{eq1.3}
follows immediately by duality.
Then \eqref{eq1.3} implies
\[
\left\|\sum \xi_k\otimes y_k\right\|_2= \left(\sum\|y_k\|^2_2\right)^{1/2} \le 4\cdot 2^{\frac1p},
\]
and we obtain $C_2(x)\le   2^{2+\frac1p} |||x|||_p$, completing step 2.
\end{proof}

\begin{rem}\label{rem1.22}
 Conversely we have
\begin{equation}\label{eq1.4}
 |||x|||_p \le C_2(x).
\end{equation}
Indeed, if $C_2(x)<1$ then $x_k = (f^{\frac1p-\frac12}y_k + y_kf^{\frac1p-\frac12})/2$ with
\[
 \left\|\sum \xi_k\otimes y_k\right\|_2 = \left(\sum\|y_k\|^2_2\right)^{1/2} < 1.
\]
Let now
\[
 a_k = (y_kf^{\frac1p-\frac12})/2\quad \text{and}\quad b_k = (f^{\frac1p-\frac12}y_k)/2.
\]
We have (recall the notation $|T| = \sqrt{T^*T}$)
\[
 2\left(\sum a^*_ka_k\right)^{1/2} = \left|\left(\sum y^*_ky_k\right)^{1/2} f^{\frac1p-\frac12}\right|
\]
and hence (setting $\frac1r = \frac1p-\frac12$) by H\"older
\begin{align*}
 2\left\|\left(\sum a^*_ka_k\right)^{1/2}\right\|_p &\le \left\|\left(\sum y^*_ky_k\right)^{1/2}\right\|_2\|f^{\frac1p-\frac12}\|_r\\
&< 1.
\end{align*}
Similarly $\|(\sum b_kb^*_k)^{1/2}\|_p < 1/2$. Thus we obtain $|||x|||_p < 1$. By homogeneity this proves \eqref{eq1.4}. 
\end{rem}

\begin{proof}[Proof of Step \ref{stp3}]
 
Fix $\vp>0$. Let $y_k$ be such that $x_k = (f^{1/r}y_k+y_kf^{1/r})/2$ with
\[
 \left\|\sum \xi_k\otimes y_k\right\|_2 < C_2(x)(1+\vp).
\]
Let us assume that $(M,\tau)$ is $M_n$ equipped with usual trace. We will use the orthonormal basis for which $f$ is diagonal with coefficients denoted by $(f_i)$. We have then
\[
 (y_k)_{ij} = 2(f^{\frac1r}_i + f^{\frac1r}_j)^{-1}(x_k)_{ij}.
\]
We define $y_k(\theta)$ by setting
\[
 y_k(\theta)_{ij} = 2(f^{\frac\theta{r}}_i + f^{\frac\theta{r}}_j)^{-1} (x_k)_{ij}.
\]
Note that $y_k(0) = x_k$ while $y_k(1) = y_k$.
Let $T(\theta) = \sum \xi_k\otimes y_k(\theta)$. 

\n We claim that if $\frac1q = \frac{1-\theta}p + \frac\theta2$ and $1\le p<q\le 2$
\begin{equation}\label{eq1.6}
\|T(\theta)\|_q \le c\|T(0)\|^{1-\theta}_p \|T(1)\|^\theta_2
\end{equation}
for some constant $c$ depending only on $p$ and $q$. As observed in \cite{JP}, when $p>1$, this is easy to prove using the boundedness of the triangular projection on $S_p$. The case $p=1$ is a consequence of Theorem 1.1a in \cite{JP} (the latter uses   \cite[Th.~4.5]{P3}).
See \S \ref{sec10} for a detailed justification.

\n Therefore we obtain
\[
C_q(x) \le \|T(\theta)\|_q \le 2^{1-\theta} c C_p(x)^{1-\theta} (C_2(x)(1+\vp))^\theta,
\]
i.e.\ we obtain Step \ref{stp3} in the matricial case. Note 
  that the argument works assuming merely that the density 
$f$ has finite spectrum.
\end{proof}
\begin{proof}[Proof of  Theorem \ref{thm1.1}]
   Combining the 3 steps, we have already indicated the proof in the case
   $M=M_n$ or assuming merely that the density 
$f$ has finite spectrum.
We will now prove the general semi-finite case.
   We  return to Step 2 . We claim that for any $\delta>0$ we can find $(x'_k)$ such that $|||(x_k)-(x'_k)|||_p < \delta |||x|||_p$ and such that 
\[
 C_2(x') \le 2\cdot 2^{\frac1p}(1+\delta) |||x'|||_p,
\] where the definition of $C_2(x')$ is now restricted
to densities with finite spectrum.

Indeed, one may assume by homogeneity that $0<|||x|||_p<1$. Let $r$ be defined by $\frac1r = \frac1p - \frac12$. Let $\delta' = (\delta/n)|||x|||_p$. Then let $f,y_k,\dots$ be as in the above proof of Step~2 and let $g\in {\cl D}$ be an element with finite spectrum such that $\|f^{\frac1r}-g^{\frac1r}\|_r < (2\cdot 2^{\frac1p})^{-1} \delta'$. Note that $g$ exists by the semi-finiteness of $\tau$. Then let
\[
 x'_k = (g^{\frac1r} y_k + y_kg^{\frac1r})/2.
\]
Note that (by H\"older) $\|x'_k-x_k\|_p<\delta'$ and hence (assuming $p\ge 1$) $|||x-x'|||_p < \delta|||x|||_p$.

We now observe that the proof of Step~3 applies if we replace $(x,f)$ by $(x',g)$. Thus if we apply the three steps to $x'$ we obtain for some constant $C_4$
\[
 |||x'|||_p \le C_4C_p(x')  = C_4\left\|\sum \xi_k \otimes x'_k\right\|_p.
\]
But since $(x'_k)$ is an arbitrary close perturbation of $(x_k)$ in $L_p$-norm, we conclude that $(K_p)$ holds.\end{proof}

\begin{rem}\label{rem1.3}
In Theorem \ref{thm1.1}, the assumption that $(\xi_k)$ is orthonormal in $L_2(\varphi)$ 
(that is only used in Step 2) can be replaced by the following one:\ for any finite sequence $y = (y_k)$ in $L_2(M,\tau)$ we have
\begin{equation}\label{eq1.5}
 \left\|\sum \xi_k\otimes y_k\right\|_{L_2(\varphi\times\tau)} \le \left(\sum \|y_k\|^2_2\right)^{1/2}.
\end{equation}
The proof (of Step \ref{stp2}) for that case is identical.
\end{rem}

 Assume for simplicity that $(M,\tau)$ is $M_n$ equipped with its usual trace.

Let $S = \sum \xi_k\otimes x_k$, $x_k\in M_n$. Equivalently $S = [S_{ij}]$ with $S_{ij}\in L_2(\varphi)$. Consider $f\in {\cl D}$. The proof of Step \ref{stp3} becomes straightforward if \eqref{eq1.6} holds.
In the case $p\ge 1$, we invoked \cite{JP} to claim that \eqref{eq1.6} is indeed true, but we do not know whether it still holds when $0<p<1$.
Nevertheless, there is a situation when \eqref{eq1.6} is easy to check, when the following condition $(\gamma',\gamma'')$ holds:

\vspace{0.25cm}

\n {\bf Condition $(\gamma',\gamma'')$:} Let $\gamma',\gamma''$ be positive numbers. We say that $S = \sum \xi_k \otimes x_k$ satisfies the condition $(\gamma',\gamma'')$ if we can find $f$ in ${\cl D}$ and $y_k$ such that $x_k = (f^{\frac1r} y_k + y_kf^{\frac1r})/2$ and such that $T = \sum \xi_k\otimes y_k$ satisfies \emph{simultaneously} the following two bounds
\begin{align}\label{eq1.7}
 \|T\|_2 &\le \gamma'C_2(x)\\
\label{eq1.8}
\|(1\otimes f^{\frac1r}) T\|_p &\le \gamma''\|S\|_p.
\end{align}
If we set $F = 1\otimes f$, we can rewrite \eqref{eq1.8} as
\begin{equation}\label{eq1.9}
 \|F^{\frac1r}T\|_p \le \gamma''\|F^{\frac1r}T + TF^{\frac1r}\|_p /2,
\end{equation}
and hence by the triangle inequality (or its analogue for $p<1$), since $S=(F^{\frac1r}T + TF^{\frac1r})/2$,
we have automatically for a suitable $\gamma'''$ (depending only on $\gamma''$ and $p$)
\begin{equation}\label{eq1.8+}
\|T F^{\frac1r} \|_p \le \gamma'''\|S\|_p.
\end{equation}
\begin{rem}
 The reason why condition $(\gamma',\gamma'')$ resolves our problem is that the one-sided version of \eqref{eq1.6} is quite easy:\ we have
\begin{equation}\label{eq1.10}
 \|F^{-\frac\theta{r}}S\|_q \le \|S\|^{1-\theta}_p \|F^{-\frac1r}S\|^\theta_2.
\end{equation}
Indeed, if we let $T = F^{-\frac1r}S$ then \eqref{eq1.10} becomes
\begin{equation}\label{eq1.11}
 \|F^{\frac{1-\theta}r}T\|_q \le \|F^{\frac1r}T\|^{1-\theta}_p \|T\|^\theta_2
\end{equation}
and the latter holds by Lemma \ref{lem1.5} below.
\end{rem}

\begin{thm}\label{thm1.4}
Let $(\xi_k)$ be a sequence in $L_2(\varphi)$ orthonormal or merely satisfying \eqref{eq1.5}. Let $0<p<q<2$. Then, if we assume the condition $(\gamma',\gamma'')$
(as above but for any $S$), the implication $(K_q)\Rightarrow (K_p)$ holds, where the resulting constant $\beta_p$ depends on $p,q,\beta_q$ and also on $\gamma',\gamma''$.
\end{thm}

For simplicity we will prove this again assuming that $(M,\tau)$ is $M_n$ equipped with its usual trace. See the above proof of Theorem \ref{thm1.1} for indications on how to check the general case.

\begin{rem}
 If $(K_p)$ holds, then there are constants $(\gamma',\gamma'')$ depending only on $p$ such that the condition $(\gamma',\gamma'')$ holds. Indeed by the above proof of Step \ref{stp2} we have
\[
 2x_k= f^{\frac1r} y_k + y_kf^{\frac1r}
\]
with
\[
 \left(\sum \|y_k\|^2_2\right)^{1/2} \le C''|||x|||_p
\]
and hence by H\"older and  \eqref{eq1.5}
\[
 \|F^{\frac1r}T\|_p \le \|T\|_2 \le \left(\sum \|y_k\|^2_2\right)^{1/2} \le C''|||x|||_p.
\]
Now if $(K_p)$ holds we have
\[
 |||x|||_p \le \beta_p\|S\|_p = \beta_p/2\|F^{\frac1r}T + TF^{\frac1r}\|_p 
\]
therefore we find $\|F^{\frac1r}T\|_p \le C''\beta_p/2\|F^{\frac1r}T +TF^{\frac1r}\|_p$
i.e. \eqref{eq1.9} holds
\end{rem}

The following two lemmas will be used.

\begin{lem}\label{lem1.5}
Let $(M,\tau)$ be a generalized measure space. Consider $F\ge 0$ in $L_1(\tau)$. Assume $0<p<q<2$. Let $\frac1r = \frac1p-\frac12$ and let $\theta$ be such that $\frac1q = \frac{1-\theta}p + \frac\theta2$. Then for any $V$ in $L_2(\tau)$ we have
\begin{align*}
\|F^{\frac{1-\theta}r}V\|_q &\le  \|F^{\frac1r}V\|^{1-\theta}_p \|V\|^\theta_2\\
\intertext{and}
\|VF^{\frac{1-\theta}r}\|_q &\le \|VF^{\frac1r}\|^{1-\theta}_p \|V\|^\theta_2.
\end{align*}
\end{lem}

\begin{proof}
It suffices to show
\begin{align}\label{eq1.7+}
\|VF^{(1-\theta)(\frac1p-\frac12)}\|_q &\le \|VF^{\frac1p-\frac12}\|^{1-\theta}_p \|V\|^\theta_2
 \end{align}
 since we obtain the other inequality   by replacing $V$ by $V^*$.
 Since the complex interpolation of non-commutative $L_p$-spaces
 is valid in the whole range $0<p<\infty $ (\cite{X2}), this can be deduced
 from the 3 line lemma.
Alternatively, this also follows from H\"older's inequality, together with \cite{Ko}. Indeed,
\begin{align*}
\|VF^{(1-\theta)(\frac1p-\frac12)}\|_q &= \||V|F^{(1-\theta)(\frac1p-\frac12)}\|_q = \|\; |V|^\theta |V|^{1-\theta} F^{(1-\theta)(\frac1p-\frac12)}\|_q\\
\intertext{and hence by H\"older (recall $\frac1q=\frac{1-\theta}p+\frac\theta2$)}
&\le \|V\|^\theta_2 \||V|^{1-\theta} F^{(\frac1p-\frac12)(1-\theta)}\|_{\frac{p}{1-\theta}}.
\end{align*}
But by \cite{Ko} (see also \cite{Ar}) we have
\[
 \|\; |V|^{1-\theta} F^{(\frac1p-\frac12)(1-\theta)}\|_{\frac{p}{1-\theta}} \le \|\; |V| F^{\frac1p-\frac12}\|^{1-\theta}_p
\]
and hence we obtain \eqref{eq1.7+}. \end{proof}

\begin{lem}[\cite{JP}]\label{lem1.6} Let $Q_j$ $(j=1,\ldots,n)$ be mutually orthogonal projections
in $M$  and let $\lambda_j$ $(j=1,\ldots,n)$ be non-negative numbers.
\begin{itemize}
\item[{\rm (i)}] For any $1\le q\le \infty$ and any $x$ in $L_q(\tau)$
\[
\left\|\sum^{n}_{i,j=1} \frac{\lambda_i\vee \lambda_j}{\lambda_i+\lambda_j}Q_i xQ_j\right\|_{L_q(\tau)} \le  \frac32\|x\|_{L_q(\tau)}
\quad {\rm and}\quad
\left\|\sum^{n}_{i,j=1} \frac{\lambda_i\wedge \lambda_j}{\lambda_i+\lambda_j}Q_i xQ_j\right\|_{L_q(\tau)} \le \frac12\|x\|_{L_q(\tau)}.
\]
\item[{\rm (ii)}]  
For any $1<q<\infty$ there is a constant $t(q)$, depending only on $q$,
such that for any $x$ in $L_q(\tau)$
   \[
\left\|\sum^{n}_{i,j=1} \frac{\lambda_i}{\lambda_i+\lambda_j}Q_i xQ_j\right\|_{L_q(\tau)} \le t(q)\|x\|_{L_q(\tau)} \quad {\rm and}\quad
\left\|\sum^{n}_{i,j=1} \frac{\lambda_j}{\lambda_i+\lambda_j}Q_i xQ_j\right\|_{L_q(\tau)} \le t(q)\|x\|_{L_q(\tau)}.
\]
\item[{\rm (iii)}]  For any $s$ with $1<q<s\le \infty$,  any density $f\in \cl D$ and any $x\in L_s(\tau)$, we have
$$ \max\{ \| f^{\frac 1q-\frac 1s} x\|_q, \| x f^{\frac 1q-\frac 1s} \|_q  \} \le t(q)\|f^{\frac 1q-\frac 1s}x+xf^{\frac 1q-\frac 1s}\|_{L_q(\tau)}.
$$
\end{itemize}
\end{lem}

\begin{proof}
This  was used in \cite{JP} (see also \cite{HK,HK2} for related facts). For the convenience of the reader we sketch the argument. We may easily reduce to the case $\sum Q_j=1$.\\
(i) expresses  the fact  that \[
 \frac{\lambda_i\vee \lambda_j}{\lambda_i+\lambda_j}\quad \text{and}\quad \frac{\lambda_i\wedge \lambda_j}{\lambda_i+\lambda_j}
\]
are (completely) contractive Schur multipliers on $L_q(M_n)$ for any $1\le q\le\infty$ (see \cite{JP}).\\
 (ii):  Using a permutation of the $(Q_j)$ we may assume $0\le \lambda_1\le \lambda_2\le ...\le\lambda_n$.
 By the boundedness of the triangular projection when $1<q<\infty$
 (see the seminal paper \cite{Mat} and \cite[\S 8]{PX} for more references to the literature),  it suffices to check (ii) when $x$ is either upper or lower triangular with respect to the decomposition $(Q_j)$. More precisely it suffices to check this when either 
$x=  x^+$ or $x=  x^-$ where $x^+=\sum _{i\le j }Q_i xQ_j$
and $x^-=\sum_{i> j }Q_i xQ_j$.
But since $\lambda_i\vee \lambda_j = \lambda_j$  and $\lambda_i\wedge \lambda_j = \lambda_i$ if $i\le j$, the case when $x$ is upper  triangular (i.e. $x= x^+ $) follows from the first part. The lower
triangular case  (i.e. $x= x^- $) is similar.\\
 (iii):  By density we may reduce this to the case when $f$ has finite spectrum, so that
 $f=\sum \lambda_j Q_j$.
 Then (iii) essentially reduces to   (ii).
\end{proof}

\begin{proof}[Proof of Theorem \ref{thm1.4}]
We choose $q>1$ with $0<p<q<2$. We use the same notation as for Theorem \ref{thm1.1}. By the observations made before Theorem \ref{thm1.4}, it suffices to verify \eqref{eq1.6}. By Lemma \ref{lem1.5} applied with $V=T$ and $V=T^*$ we have
\begin{align*}
 \|F^{\frac{1-\theta}r}T\|_q &\le \|F^{\frac1r}T\|^{1-\theta}_p \|T\|^\theta_2,\\
\|TF^{\frac{1-\theta}r}\|_q &\le \|TF^{\frac1r}\|^{1-\theta}_p \|T\|^\theta_2.
\end{align*}
Let $\lambda_i = f^{\frac\theta{r}}_i$. By Lemma \ref{lem1.6} we have
\begin{equation}\label{eq1.12}
\|[(f^{\frac\theta{r}}_i + f^{\frac\theta{r}}_j)^{-1} f^{\frac1r}_iT_{ij}]\|_q \le t(q) \|F^{\frac{1-\theta}r}T\|_q
\end{equation}
and similarly
\begin{equation}\label{eq1.13}
 \|[(f^{\frac\theta{r}}_i + f^{\frac\theta{r}}_j)^{-1} T_{ij} f^{\frac1r}_j\|_q \le t(q)\|TF^{\frac{1-\theta}r}\|_q.
\end{equation}
Note that we have
\[
 T(\theta) = (f^{\frac\theta{r}}_i + f^{\frac\theta{r}}_j)^{-1} (f^{\frac1r}_i + f^{\frac1r}_j)T_{ij}  \]
Therefore by the triangle inequality  and \eqref{eq1.12}, \eqref{eq1.13}
\[
 \|T(\theta)\|_q \le t(q)(\|F^{\frac1r}T\|^{1-\theta}_p + \|TF^{\frac1r}\|^{1-\theta}_p)\|T\|^\theta_2
\]
and hence by condition $(\gamma',\gamma'')$ 
\begin{align*}
 \|T(\theta)\|_q &\le t(q)((\gamma'')^{1-\theta}+(\gamma''')^{1-\theta}) \|(F^{\frac1r}T +TF^{\frac1r})/2\|^{1-\theta}_p \|T\|^\theta_2\\
&\le t(q)((\gamma'')^{1-\theta}+(\gamma''')^{1-\theta}) \|S\|^{1-\theta}_p \|T\|^\theta_2.
\end{align*}
Therefore we obtain \eqref{eq1.6}. By condition $(\gamma',\gamma'')$ we have $\|T(1)\|_2=\|T\|_2 \le \gamma'C_2(x)$, and also $C_q(x) \le \|T(\theta)\|_q$ so we conclude that Step \ref{stp3} holds.
\end{proof}

\begin{rem}   Theorem \ref{thm1.1} implies as a special case the
following fact possibly of independent interest:  if for some $0<q<2$ we have
\[
(C_q)\quad \left\{
\begin{array}{l}
 \exists C\ \forall a_k\in L_p(\tau)\\
\left\|\left(\sum a^*_ka_k\right)^{1/2}\right\|_{L_q(\tau)} \le C\left\|\sum \xi_k\otimes a_k\right\|_{L_q(\varphi\times\tau)}
\end{array}\right.
\]
then $(C_p)$ holds (for a different constant $C$) for all $p$ with $0<p<q$. In this case Step \ref{stp3} is easy to verify (only right multiplication appears in this case).

\end{rem}  

\begin{rem}
If $0<p\le 2$, the converse inequality to $(K_p)$ is valid assuming that
$\varphi(1)=1$ and $\xi_k\in L_2(N,\varphi)$ is orthonormal or satisfies \eqref{eq1.5}.

Indeed, for any $t\ge 0$ in $N\otimes M$ since $p/2\le 1$ and $\varphi(1)=1$, by the operator concavity of
$t\mapsto t^{p/2}$ (see \cite[p. 115-120]{Bh}) , 
 we have 
\[
 \|t\|_{p/2} \le \|{\bb E}^M(t)\|_{p/2}
\]
and hence, if $S = \sum \xi_k\otimes x_k$, we have
$$
 \|S\|_p = \|S^*S\|^{1/2}_{p/2}  \le \|{\bb E}^M(S^*S)\|^{1/2}_{p/2}\\
 \le \left\|\left(\sum x^*_kx_k\right)^{1/2}\right\|_p,
$$
and similarly
\[
 \|S\|_p \le \left\|\left(\sum x_kx^*_k\right)^{1/2}\right\|_p.
\]
From this we easily deduce\[
 \|S\|_p \le c(p)|||x|||_p
\]
where $c(p) = 1$ if $1\le p\le 2$ and $c(p) = 2^{\frac1p-1}$ if $0<p\le 1$.
The preceding remark shows that the assumption that $\varphi$ is finite cannot be removed.
\end{rem}

\begin{rem} To extend Theorem \ref{thm1.1} to the case $0<p<1$ the difficulty lies in Step \ref{stp3}, or in proving
a certain form of H\"older inequality such as  \eqref{eq1.6}. Note that a much \emph{weaker} estimate allows to conclude:\\
It suffices to show that there is a function $\vp\to\delta(\vp)$ tending to zero with $\vp>0$ such that when $f\in {\cl D}$ we have $(\alpha=\frac1p - \frac12=\frac1r)$ $(1<q<2)$:
$$    [\|x\|_2 \le 1, \|f^\alpha x + xf^\alpha\|_p\le \vp]\Rightarrow
 \|f^{\alpha(1-\theta)} x + xf^{\alpha(1-\theta)}\|_q \le \delta(\vp).
$$
This might hold even if Step \ref{stp3} poses a problem.
\end{rem}

In the case $2\le q<\infty$, the formulation of $(K_q)$ must be changed. When $2<q<\infty$, and $x=(x_k)$ is a finite sequence in $L_q(\tau)$, we set
\[
 |||x|||_q \overset{\scriptstyle\text{def}}{=} \max\left\{\left\|\left(\sum x^*_k x_k\right)^{1/2} \right\|_q, \left\|\left(\sum x_kx^*_k\right)^{1/2}\right\|_q\right\}.
\]
We will then say (when $2<q<\infty$) that $(\xi_k)$ satisfies $(K_q)$ if there is a constant $\beta_q$ such that for any such $x=(x_k)$ we have
\[
 \left\|\sum \xi_k\otimes x_k\right\|_q \le \beta_q|||x|||_q.
\]
By \cite{LP1}, this holds when $(\xi_k)$ are the Rademacher functions on [0,1]. 

In operator space theory, all $L_q$-spaces, in particular the Schatten class $S_q$, can be equipped with a ``natural'' operator space structure (see \cite[\S 9.5 and 9.8]{P2}). Let $C_q$ (resp.\ $R_q$) denote the closed span of $\{e_{i1}\mid i\ge 1\}$ (resp.\ $\{e_{1j}\mid j\ge 1\}$) in $S_q$. We denote by $C_{q,n}$ and $R_{q,n}$ the corresponding $n$-dimensional subspaces. By definition, the ``sum space'' $R_q+C_q$ is the quotient space $(R_q\oplus C_q)/N$ where $N =\{(x,-{}^tx)\mid x\in R_q\}$. We define similarly $R_{q,n}+C_{q,n}$. The intersection $R_q\cap C_q$ is defined as the subspace $\{x,{}^tx\}$ in $R_q\oplus C_q$. Here the direct sums are meant (say) in the operator space sense, i.e.\ in the $\ell_\infty$-sense. Let us denote by $\text{Rad}(n,q)$ (resp.\ $\text{Rad}(q)$) the linear span of the first $n$ (resp.\ the sequence of all the) Rademacher functions in $L_q([0,1])$. The operator space structure induced on the space $\text{Rad}(q)$ is entirely described by the non-commutative Khintchine inequalities (see \cite[\S   9.8]{P2}):\ the space $\text{Rad}(q)$ is completely isomorphic to $R_q+C_q$ when $1\le q\le 2$ and to $R_q\cap C_q$ when $2\le q<\infty$. The case $0<q<1$ is open.

Note that $\text{Rad}(q,n)$ is an $n$-dimensional subspace of $\ell^{2^n}_q$, so  $(Kh_q)$ implies that $R_{q,n}+C_{q,n}$ uniformly embeds into $\ell^{2^n}_q$ for all $1\le q<2$. 
The next result improves significantly the dimension of the embedding.

 First recall that two operator spaces $E,F$ are called completely $c$-isomorphic if there is an isomorphism $u\colon \ E\to F$ such that $\|u\|_{cb}\|u^{-1}\|_{cb}\le c$. 
\begin{thm}\label{thm1.7}
 Let $1\le q<2$. For any $n$, there is a subspace of $\ell^n_q$ with dimension $k=[n^{1/2}]$ that is completely $c$-isomorphic to $R_{q,k} +C_{q,k}$ where $c$ is a constant depending only on $q$.
\end{thm}

\begin{proof}
 By \cite{Har}, we know that there is a subset of $[1,e^{it},\ldots, e^{int}]$ with cardinality $k=[n^{1/2}]$ such that the corresponding set $\{\xi_1,\ldots, \xi_k\}$ satisfies $(K_q)$ for all $q$ such that $2\le q\le 4$ and hence by duality for all $q$ such that $4/3 \le q\le 2$. By Theorem \ref{thm1.1}, the same set satisfies $(K_p)$ for all $1\le p\le 2$ (with a constant $\beta_p$ independent of $n$).
\end{proof}

\begin{prb}
 What is the correct estimate of $k$ in Theorem \ref{thm1.7}? In particular is it true for $k\sim [n^\alpha]$ with $\alpha$ any number in (0,1) instead of $\alpha=\frac12$? Is it true for $k\sim [\delta n]$ $(0<\delta<1)$?
\end{prb}
The case $q=1$ is particularly interesting.
It is natural to expect a positive answer with $k$ proportional to $n$ by analogy with the Banach space case (see \cite{FLM}). One could even dream of an operator
space version of the Kashin decomposition (cf. \cite{Sz}) !
Another bold conjecture would be     the operator space generalization
of Schechtman's 
and Bourgain, Lindenstrauss and Milman's results, as refined by   Talagrand in  \cite{Ta}:
\begin{prb} Let $E$ be an $n$-dimensional operator subspace of $S_1$.
Assume that for some $p>1$ $E$ embeds $c$-completely isomorphically   into $S_p$. Is it true that $E$ can then be embedded $c'$-completely isomorphically
(the constant $c'$ being a function of $p$ and $c$) into $S_1^{2n}$ ?
\end{prb}

\section{Conditional expectation variant}\label{sec2}

Again we consider $(\xi_k)$ in $L_2(N,\varphi)$ and coefficients $(x_k)$ in $L_2(M,\tau)$, but, in addition, we give ourselves a von~Neumann subalgebra ${\cl M}\subset M$ such that $\varphi_{|{\cal M}}$ is semi-finite and we denote by $E\colon \ M\to {\cl M}$ the conditional expectation with respect to ${\cl M}$. Recall that $E$ extends to a contractive projection (still denoted abusively by $E$) from $L_q(M,\tau)$ onto $L_q({\cl M},\tau)$ for all $1\le q\le \infty$.

Consider $x=(x_k)$ with $x_k\in L_q(\tau)$. In this section, we define
\[
|||x|||_{q,{\cl M}} = \inf_{x_k=a_k+b_k}\left\{\left\|\left(E \sum a^*_ka_k\right)^{1/2}\right\|_q + \left\|\left( E \sum a_ka^*_k\right)^{1/2}\right\|_q\right\}.
\]
We then define again
\[
 C_{q,{\cl M}} (x)= \inf\left\{\left\|\sum \xi_k \otimes y_k\right\|_q\right\}
\]
where the infimum runs over all $y_k$ in $L_q(\tau)$ such that there is $f$ in ${\cl D}\cap L_1({\cl M})$ such that $x_k = (f^{1/r}y_k + y_kf^{1/r})/2$.

Then the proof described in \S \ref{sec1} extends with no change to this situation and shows that if there is $\beta_q({\cl M})$ such that for all finite sequences $x=(x_k)$ in $L_q(\tau)$ we have
\begin{equation}\label{eq2.1}
 |||x|||_{q,{\cl M}} \le \beta_q({\cl M}) \left\|\sum \xi_k\otimes x_k\right\|_q
\end{equation}
then for any $p$ with $1\le p<q$ there is a constant $\beta_p({\cl M})$  such that \eqref{eq2.1} holds for any $x=(x_k)$ in $L_p(\tau)$ when $q$ is replaced by $p$. The main new case we have in mind is the case when $L_q(M,\tau) = S_q$ (Schatten class) and ${\cl M}$ is the subalgebra of \emph{diagonal} operators (so that $L_q({\cl M})\simeq \ell_q)$ on $\ell_2$.

Thus we state for future reference:

\begin{thm}\label{thm2.1}
 Both Theorem \ref{thm1.1} and Theorem \ref{thm1.4} remain valid when $|||\cdot|||_q$ is replaced by $|||\cdot|||_{q,{\cl M}}$.
\end{thm}

\begin{proof}
 The verification of this assertion is straightforward. Note that the conditional expectation $E$ satisfies $E(ax b) = aE(x)b$ whenever $a,b$ are in ${\cl M}$ and $x$ in $L_q(M,\tau)$. This is used to verify Step \ref{stp2}. The proof of the other steps require no significant change.
\end{proof}

 Let $L_q(M,\tau) = S_q$ (Schatten $q$-class) and let ${\cl M}\subset B(\ell_2)$ be the subalgebra of diagonal operators with conditional expectation denoted by $E$. Consider a family $x=(x_{ij})$ with $x_{ij}\in S_q$.
We will denote by $(x_{ij})_{k\ell}$ the entries of each matrix $x_{ij}\in S_q$, and we set
$$
 \hat x_{ij} = (x_{ij})_{ij} e_{ij} \quad 
 {\rm and}\quad 
\hat x  = (\hat x_{ij}).$$
Let us denote
\[
 \|x||_{R_q} \overset{\rm def}{=} \left\|\left( \sum_{ij} x_{ij}x^*_{ij}\right)^{1/2}\right\|_q\quad \text{and}\quad \|x\|_{C_q} \overset{\rm  def}{=} \left\|\left(\sum_{ij} x^*_{ij}x_{ij}\right)^{1/2}\right\|_q.
\]
\begin{lem} For any $q\ge 1$ we have
\begin{align}\label{eq2.2}
 \|\hat x\|_{R_q} &\le \|x\|_{R_q}\quad \text{and}\quad \|\hat x\|_{C_q} \le \|x\|_{C_q},\\
\intertext{and consequently}
\label{eq2.3}
|||\hat x|||_q &\le |||x|||_q.
\end{align}
\end{lem}
\begin{proof}
Indeed, this follows from the convexity of the norms involved and the identity
\[
 \hat x_{ij} = \int\overline{ z'_i} \overline{ z''_j} D(z') x_{ij}D(z'')\ dm(z') dm(z'')
\]
where $z' = (z'_i)$, $z'' = (z''_j)$ denote elements of ${\bb T}^{\bb N}$ equipped with its normalized Haar measure $m$. 
\end{proof}

 Now consider $\lambda_{ij}\in {\bb C}$.
 We define \begin{equation}\label{lambda}
 [\lambda]_p  = \inf\left\{\left(\sum_i\left(\sum_j |a_{ij}|^2\right)^{p/2} + \sum_j \left(\sum_i |b_{ij}|^2\right)^{p/2}\right)^{1/p}\right\}
\end {equation}

where the inf runs over all possible decompositions $\lambda_{ij} = a_{ij}+b_{ij}$.
 
 \begin{lem}\label{lam}
  Let $\hat x_{ij} = \lambda_{ij}e_{ij}$ (i.e. $\lambda_{ij} = (x_{ij})_{ij}$). We have
\[
 \|\hat x\|_{C_q} =  \left(\sum\nolimits_j \left(\sum\nolimits_i |\lambda_{ij}|^2\right)^{q/2}\right)^{1/q} \quad \text{and}\quad \|\hat x\|_{R_q} = \left(\sum\nolimits_i \left(\sum\nolimits_j |\lambda_{ij}|^2\right)^{q/2} \right)^{1/q}.
\]Moreover, for any $0<q<\infty$
\[
|||\hat x|||_{q,{\cl M}}= [\lambda]_q \le |||x|||_{q,{\cl M}}.\]
  where ${\cl M}$ is the subalgebra of \emph{diagonal} operators   on $\ell_2$.

\end{lem}
 \begin{proof} The first assertion is an immediate calculation.
 We now show  that for any $0<q<\infty$
\begin{equation}\label{eq2.5}
 [\lambda]_q \le |||x|||_{q,{\cl M}}.
\end{equation}
Indeed, let us denote
\[
 \|x\|_{C_q,{\cl M}} = \left\|\left(E\sum_{ij}x^*_{ij}x_{ij}\right)^{1/2}\right\|_q \quad\text{and} \quad \|x||_{R_q,{\cl M}} = \left\|\left( E\sum x_{ij}x^*_{ij}\right)^{1/2}\right\|_q.
\]
Then a simple verification shows that
$$
 \|x\|_{C_q,{\cl M}}  = \left(\sum_k \left(\sum_{ij\ell} |(x_{ij})_{\ell k}|^2\right)^{q/2}\right)^{1/q}
 \ge \left(\sum_k \left(\sum_\ell |(x_{\ell k})_{\ell k}|^2\right)^{q/2}\right)^{1/q}$$
 {and similarly we find} 
$$\|x\|_{R_q,{\cl M}}  \ge \left(\sum_\ell \left(\sum_k |(x_{\ell k})_{\ell k}|^2\right)^{q/2}\right)^{1/q}.
$$
Then \eqref{eq2.5} follows immediately. In particular 
$ [\lambda]_q\le |||\hat x|||_{q,{\cl M}}$, and since the converse is immediate we obtain
$|||\hat x|||_{q,{\cl M}}=[\lambda]_q$.
 \end{proof}

\begin{rem}
It seems worthwhile to point out that \eqref{eq2.2} is no longer valid when $0<q<1$ (even with a constant). Indeed, restricting to the case when $x_{ij} = 0$ $\forall i\ne j$, these inequalities imply
\begin{equation}\label{eq2.4}
 \left(\sum_j |(x_{jj})_{jj}|^q\right)^{1/q} \le \left\|\left(\sum\nolimits_j x^*_{jj}x_{jj}\right)^{1/2}\right\|_q.
\end{equation}
Now let us consider the case
\[
 x_{jj} = \sum^n_{k=1} e_{jk}.
\]
On one hand we have
$x^*_{jj}x_{jj}=nP$ where $P$ is the rank one orthogonal projection onto $n^{-1/2}\sum e_k$, so that
$(\sum x^*_{jj}x_{jj})^{1/2}=nP$ and hence $\big\|\big(\sum x^*_{jj}x_{jj}\big)^{1/2}\big\|_q=n$.
But on the other hand $(x_{jj})_{jj}=1$ and hence
\[
 \left(\sum|(x_{jj})_{jj}|^q\right)^{1/q} = n^{1/q}.
\]
This shows that \eqref{eq2.4} and hence \eqref{eq2.2} fails for $q<1$. The same example shows (a fortiori) that \eqref{eq2.3} also fails for $q<1$.
\end{rem}

\begin{rem}\label{rem2.3}
 Let $j\colon\ L_p(M,\tau)\to L_p(M',\tau')$ be an isometric embedding. Let $y_k = j(x_k)$. Then clearly
\[
 \int\left\|\sum r_k(t)y_k\right\|^p_p dt = \int \left\|\sum r_k(t)x_k\right\|^p_p dt.
\]
However, when $0<p<1$, we do not see how to prove that there is a constant $C$ such that
\[
 |||(x_k)|||_p \le C|||(y_k)|||_p,
\]
although when $p\ge 1$ this holds with $C=1$ using a conditional expectation.
\end{rem}

This may be an indication that   $(Kh_p)$ does not hold for $0<p<1$, at least  in the same form as for $p\ge 1$.

\section{The case $\pmb{0<p<1}$}\label{sec3}

In $(K_q)$, we may consider the case when $L_q(\tau) = S_q$ (Schatten $q$-class) and the sequence $x=(x_k)$ is of the form $x_{ij} = \lambda_{ij}e_{ij}$ with $\lambda_{ij} \in {\bb C}$. For this special case, the approach used in the preceding section works for all $0<p<q$. Thus,
 we obtain:

\begin{thm}\label{thm3.1}
 Let $(\vp_{ij})$ be an i.i.d sequence of $\{+1,-1\}$-valued random variables on a probability space with ${\bb P}(\vp_{ij}=\pm 1)=1/2$. Then for any $0<p<1$ there is a constant $\beta_p$ such that
$$
[\lambda]_p  \le \beta_p \left(\int\left\|\sum \vp_{ij}\lambda_{ij}e_{ij}\right\|^p_{S_p} d{\bb P}\right)^{1/p}\\
$$ {where}
$[\lambda]_p$
is defined in \eqref{lambda}.  
\end{thm}
\begin{rk} Of course, by \cite{LP1,LPP}, the case $1\le p\le 2$ is already known.
\end{rk}
\begin{rem}\label{rem3.2}
 Since $S_p$ is $p$-normed when $0<p<1$, the converse inequality is obvious:\ we have
\[
 \|a\|_p \le \left(\sum_i \left\|\sum_j a_{ij}e_{ij}\right\|^p_p\right)^{1/p} = \left(\sum_i\left(\sum_j |a_{ij}|^2\right)^{p/2}\right)^{1/2}
\]
and similarly $\|b\|_p \le (\sum_j(\sum_i|b_{ij}|^2)^{p/2})^{1/2}$. Therefore
\[
 \sup_{\vp_{ij}=\pm 1} \left\|\sum \vp_{ij}\lambda_{ij}e_{ij}\right\|_p \le [\lambda]_p.
\]
By well known general results (cf.\cite{Ka}) this allows us to formulate the 
\end{rem}

\begin{cor}\label{cor3.3}
 Let $\lambda_{ij}\in {\bb C}$ be arbitrary complex scalars. The following are equivalent.
\begin{itemize}
\item[\rm (i)] The matrix $[\vp_{ij}\lambda_{ij}]$ belongs to $S_p$ for almost all choices of signs $\vp_{ij} = \pm 1$.
\item[\rm (ii)] Same as (i) for all choices of signs.
\item[\rm (iii)] There is a decomposition $\lambda_{ij} = a_{ij}+b_{ij}$ with
\[
\sum_i\left(\sum_j|a_{ij}|^2\right)^{p/2} <\infty \quad \text{and}\quad \sum_j\left(\sum_i |b_{ij}|^2\right)^{p/2}<\infty
\]
i.e.\ in short $[\lambda]_p<\infty$.
\end{itemize}
\end{cor}

\begin{rem}
Note that when $0<p<1$, the spaces $S_p$ or $L_p(\tau)$ are $p$-normed, i.e.\ their norm satisfies for any pair of elements $x,y$
\begin{equation}\label{eq3.1}
 \|x+y\|^p \le \|x\|^p + \|y\|^p.
\end{equation}
\end{rem}

\begin{rem}\label{rem3.4}
 Assume here that $0<p\le 2$. Note that $[\lambda]_p<1$ implies that there is a sequence $f_i>0$ with $\sum f_i\le 1$ such that, if we set $\frac1r = \frac1p-\frac12$ we have
\begin{equation}\label{eq3.2}
 \left(\sum_{ij}|(f^{\frac1r}_i + f^{\frac1r}_j)^{-1} \lambda_{ij}|^2\right)^{1/2} \le 2.
\end{equation}
Indeed, we have $\lambda_{ij} = a_{ij}+b_{ij}$ with 
$(\sum\limits_i(\sum\limits_j|a_{ij}|^2)^{p/2})^{1/p} + (\sum\limits_j(\sum\limits_i |b_{ij}|^2)^{p/2})^{1/p}<1$
we then set $a_i = (\sum\limits_j|a_{ij}|^2)^{p/2}$ and  $b_j = (\sum\limits_i|b_{ij}|^2)^{p/2}$ so that $(\sum a_i) + (\sum b_j) < 1$. Note that:
\[
 |(a^{\frac1r}_i + b^{\frac1r}_j)^{-1} (\lambda_{ij})| \le a^{-\frac1r}_i |a_{ij}| + |b_{ij}|b^{-1/r}_j
\]
and hence by H\"older
\begin{align*}
 \left(\sum_{ij} |(a^{\frac1r}_i + b^{\frac1r}_j)^{-1} \lambda_{ij}|^2\right)^{1/2} &\le \left(\sum |a^{-1/r}_i a^{1/p}_i|^2\right)^{1/2}\\
&\quad + \left(\sum |b^{-1/r}_j b^{1/p}_j|^2\right)^{1/2}\\
&= \left(\sum a_i\right)^{1/2} + \left(\sum b_j\right)^{1/2} \le 2.
\end{align*}
Let $f_i = a_i+b_i$. Then $\sum f_i < 1$,  \eqref{eq3.2} holds and, if we  perturb  $f_i$ slightly, we may assume $f_i>0$ for all $i$.
$\hfill\square$
\end{rem}

We will use the following well known elementary fact.

\begin{pro}\label{pro3.5}
 Let $X$ be a $p$-normed space $(0<p\le 1)$, i.e.\ we assume
\begin{equation}
\|x+y\|^p \le \|x\|^p + \|y\|^p.\tag*{$\forall x,y\in X$}
\end{equation}
Then there is a constant $\chi_p$ such that for any finite sequence $(x_k)$ in $X$ and any sequence of real numbers $(\alpha_k)$ we have
\begin{equation}\label{eq3.3}
 \left\|\sum \alpha_k r_k x_k\right\|_{L_p(X)} \le \chi_p \sup_k|\alpha_k| \left\|\sum \vp_kx_k\right\|_{L_p(X)}.
\end{equation}
Here $(r_k)$ denote the Rademacher functions on $[0,1)$ and $L_p(X) = L_p([0,1]; X)$.
\end{pro}

\begin{proof}
If $\alpha_k\in \{-1,1\}$, we have equality in \eqref{eq3.3} with $\chi_p=1$. If $\alpha_k \in \{-1,0,1\}$ we can write $\alpha_k = (\beta_k+\gamma_k)/2$ with $\beta_k\in \{-1,1\}$, $\gamma_k\in \{-1,1\}$ and then we obtain \eqref{eq3.2} (using the $p$-triangle inequality \eqref{eq3.1}) with $\chi_p = 2^{\frac1p-1}$. For the general case, we can write any $\alpha_k$ in $[-1,1]$ as a series $\alpha_k = \sum^\infty_1 \alpha_k(m) \xi_k(m)$ with $\alpha_k(m)\in \{-1,0,1\}$ and $|\xi_k(m)|\le 2^{-m}$. We then obtain \eqref{eq3.3} with 
\begin{align*}
\chi_p &= 2^{\frac1p-1} \left(\sum\nolimits^\infty_1 2^{-mp}\right)^{1/p}\\
&= 2^{\frac1p-1} (2^p-1)^{-1/p}.\qquad\qed
\end{align*}
\renewcommand{\qed}{}\end{proof}

\begin{proof}[Proof of Theorem \ref{thm3.1}]
Let $S = \sum \vp_{ij}\lambda_{ij}e_{ij}$ and let $x_{ij} =  e_{ij}\lambda_{ij}$. 
We already know by \cite{LP1,LPP} the case $1\le p<2$. We will 
show that the condition $(\gamma',\gamma'')$ holds, and hence that Theorem \ref{thm3.1} follows
from Theorem \ref{thm1.4}. 

Again we assume that $M=M_n$. Let ${\cl M}\subset M$ be the subalgebra of diagonal matrices with associated conditional expectation
denoted by $E$.
Let $x=(x_{ij})$. Then  by Lemma \ref{lam}
for any $0<q<2$ we have
$$[\lambda]_q =  |||x|||_{q, {\cal M}  }.$$
So the inequality in Theorem \ref{thm3.1} boils down
to $|||x|||_{p, {\cal M}  }\le \beta_p \|\sum \vp_{ij} x_{ij}\|_p$.
We need to observe that when we run the proof of Theorem 1.6 with  $|||x|||_{p, {\cal M}  }$
in place of  $|||x|||_{p  }$ we only need to know $(K_q)$ for
a family $(y_k)$ such that each $y_k$ lies in the closure in $L_q(\tau)$
of elements in  ${\cal M} x_k {\cal M}$. When  ${\cal M}$ is the algebra
of diagonal operators, that means that $y_k$ is obtained from $x_k$ by
a Schur multiplier, so that in any case when $(x_k)$ is the family $(x_{ij})$
given as above by $x_{ij} =  e_{ij}\lambda_{ij}$, then  all the families
$(y_{ij})$ are also of the same form  i.e. we have $y_{ij} =  e_{ij}\mu_{ij}$
for some scalars $\mu_{ij}$, and for the latter we know by \cite{LP1,LPP} that
the ${\cal M}$-version of $(K_q)$ holds for $1\le q\le2$.

So we will be able to conclude if we can verify the condition $(\gamma',\gamma'')$.
We claim that for some constant $C$

\begin{equation}\label{eq3.5}
 \|f^{\frac1p-\frac12}T\|_p \le C\|f^{\frac1p-\frac12} T + Tf^{\frac1p-\frac12}\|_p
\end{equation}
where $f$ is any positive diagonal matrix and   $T=\sum \vp_{ij} y_{ij} $, with $y_{ij}$ of the form $y_{ij} =  e_{ij}\mu_{ij}$ as above.
Indeed, we have
\[
 f^{\frac1p-\frac12}_i \le f^{\frac1p-\frac12}_i + f^{\frac1p-\frac12}_j
\]
and hence, by Proposition \ref{pro3.5}, \eqref{eq3.5} holds with $C=\chi_p$.
Thus, we have condition $(\gamma',\gamma'')$ with $\gamma''=\chi_p$ and by Remark \ref{rem3.4}
we can arrange to have, say, $\gamma'=4$.
Thus, modulo the above observation,   we may view Theorem \ref{thm3.1} as a corollary to Theorem \ref{thm2.1}. 
\end{proof}

\begin{rem}\label{rem9}
 Assume $\lambda_{ij}\in L_p(M,\tau)$ (or simply $\lambda_{ij}\in S_p$)
 and let $x_{ij} = e_{ij}\otimes \lambda_{ij} \in L_p(B(\ell_2) \otimes M)$. Then, at the time of this writing,  
 we do not know whether Theorem \ref{thm3.1} remains valid  for the series $\sum \vp_{ij}e_{ij}\otimes \lambda_{ij}$, with $[\lambda]_p$ replaced by
\[
[[\lambda]]_p = \inf\left\{\left(\sum_i \left\|\left(\sum_j a^*_{ij}a_{ij}\right)^{1/2}\right\|^p_p\right)^{1/p} + \left(\sum_j \left\|\left(\sum_i b_{ij}b^*_{ij}\right)^{1/2}\right\|^p_p\right)^{1/p}\right\}
\]
where the infimum runs over all decomposition,
$
\lambda_{ij} = a_{ij}+b_{ij}$  {in} $L_p(\tau)$. By $(Kh_p)$ this clearly holds when $p\ge 1$.
 
\end{rem}

\section{Remarks on $\pmb{\sigma(q)}$-sets and $\pmb{\sigma(q)_{cb}}$-sets}\label{sec3+}

In \cite{Har} (see also \cite{Har2}) the following notion is introduced:

\begin{defn}
A subset $E\subset {\bb N}\times {\bb N}$ is called a $\sigma(q)$-set $(0<q\le \infty)$ if the system $\{e_{ij}\mid (i,j)\in E\}$ is an unconditional basis of its closed linear span in $S_q$.
\end{defn}

Equivalently, there is a constant $C$ such that for any finitely supported family of scalars $\{\lambda_{ij}\mid (i,j)\in E\}$ and any bounded family of scalars $(\alpha_{ij})$ with $\sup|\alpha_{ij}| \le 1$ we have
\[
 \left\|\sum_{(i,j)\in E} \alpha_{ij}\lambda_{ij}e_{ij}\right\|_{S_q} \le C\left\|\sum_{(i,j)\in E} \lambda_{ij}e_{ij}\right\|_{S_q}.
\]
The smallest such constant $C$ is denoted by $\sigma_q(E)$.

 The ``operator space'' version of this notion is as follows:\ $E$ is called a $\sigma(q)_{cb}$-set if there is a $C$ such that for any finitely supported family $\{\lambda_{ij}\mid (i,j)\in E\}$ in $S_q$ and any $(\alpha_{ij})$ as before we have
\[
 \left\|\sum_{(i,j)\in E} \alpha_{ij}e_{ij}\otimes \lambda_{ij}\right\|_{S_q(\ell_2\otimes\ell_2)} \le C \left\|\sum_{(i,j)\in E} e_{ij}\otimes \lambda_{ij}\right\|_{S_q(\ell_2\otimes\ell_2)}.
\]
We then denote by $\sigma^{cb}_q(E)$ the smallest such constant $C$.

It is not known whether $\sigma(q)$-sets are automatically $\sigma(q)_{cb}$-sets when $q\ne 2$. (The case $q=2$ is trivial:\ every subset $E$ is $\sigma(2)_{cb}$.) By the non-commutative Khintchine inequalities (\cite{LP1,LPP}), if $1\le q<2$, $E\subset {\bb N}\times {\bb N}$ is a $\sigma(q)$-set (resp.\ $\sigma(q)_{cb}$-set) iff there is a constant $C'$ such that for all families $\{\lambda_{ij}\mid (i,j)\in E\}$ with $\lambda_{ij}$ scalar (resp.\ $\lambda_{ij}\in S_q$) we have
\begin{align*}
[\lambda]_q &\le C'\left\|\sum_{(i,j)\in E} \lambda_{ij}e_{ij}\right\|_{S_q}\\
\intertext{$\Big($resp.}
[[\lambda]]_q &\le C'\left\|\sum_{(i,j)\in E} e_{ij}\otimes \lambda_{ij}\right\|_{S_q(\ell_2\otimes \ell_2)}\Big).
\end{align*}
The proof of Theorem \ref{thm1.1}, modified
as   in Theorem \ref{thm2.1}, yields the following complement to \cite{Har}:

\begin{thm}\label{sigmap}
 Assume $1\le p<q<2$. Any $\sigma(q)$-set 
 (resp.\ $\sigma(q)_{cb}$-set) $E\subset {\bb N}\times {\bb N}$ is a $\sigma(p)$-set
 (resp.\ $\sigma(p)_{cb}$-set). \end{thm}

\begin{cor} There is a constant $c\ge 1$ such that, for any $n$, the usual ``basis" 
$\{ e_{ij}\}$ of $S_1^n$ contains  
a $c$-unconditional subset of size $\ge   n^{3/2}$.
\end{cor}
\begin{proof} By \cite[Th. 4.8]{Har}   there is a constant $c\ge 1$
 such that, for any $n$, the set $[n]\times [n]$ contains a further (``Hankelian") subset that is a
 $\sigma(4)_{cb}$-set
(and hence by duality also $\sigma(4/3)_{cb}$) with constant $\le c$ and cardinal $\ge   n^{3/2}$.
\end{proof}
\begin{prb} What is the ``right" order of growth in the preceding statement?
 Can 3/2 be replaced by any number $<2$ ?
 \end{prb}
 
\begin{rem}
 As observed in \cite{Har}, if $2<p<q$, it is easy to show by interpolation that any $\sigma(q)$-set (resp.\ $\sigma(q)_{cb}$-set) is a $\sigma(p)$-set (resp.\ $\sigma(p)_{cb}$-set). Moreover, any such set is a $\sigma(q')$-set (resp.\ $\sigma(q')_{cb}$-set) where $q^{\prime-1} = 1 - q^{-1}$. However, the fact that e.g.\ $\sigma(q')\Rightarrow q(1)$   is new as far as we know.
\end{rem}

\section{Grothendieck-Maurey  factorization for Schur multipliers \\ ($\pmb{0<p<1}$)}\label{sec4}

Consider a bounded linear map $u\colon \ H\to L_p(\tau)$ on a Hilbert space $H$ with $0<p\le 2$. 
To avoid technicalities, we assume that the range of $u$ lies in a finite dimensional von Neumann subalgebra of $M$ on which $\tau$ is finite. When $p\ge1$, it is known that there is $f$ in $L_1(\tau)_+$ with $\tau(f)=1$ and a bounded linear map $\tilde u\colon \ H\to L_2(\tau)$ such that
\begin{equation}
 u(x) = f^{\frac1p-\frac12} \tilde u(x) + \tilde u(x)f^{\frac1p-\frac12}\tag*{$\forall x\in H$}
\end{equation}
and $\|\tilde u\| \le K_p\|u\|$ where $K_p$ is a constant independent of $u$.

In the case $p=1$, this fact is easy to deduce from the
  dual form proved  in \cite{P1} for maps
  from $M$ to $H$; the latter is often designated as the non-commutative ``little GT'' (here GT stands for Grothendieck's theorem). It is easy to deduce this statement from $(Kh_p)$ (see \cite{LPP} for more details) in the case $1\le p<2$ (note that $p=2$ is trivial). 
See \cite{LP2} for a proof that the best constant $K_p$ remains bounded when $p$ runs over [1,2]. We refer the reader to \cite{LP2,LPX} and \cite{JP} for various generalizations. 

It seems natural to conjecture that the preceding factorization of $u$ remains valid for any $p$ with $0<p<1$.  Unfortunately, we leave this open. Nevertheless, in analogy with \S \ref{sec3}, we are able to prove the preceding factorization in the special case of Schur multipliers as follows.

\begin{thm}\label{thm4.1}
 Let $0<p<1$. Let $r$ be such that $\frac1r = \frac1p-\frac12$. Consider a Schur multiplier
\[
u_\varphi\colon \ [x_{ij}] \to [x_{ij}\varphi_{ij}]
\]
where $\varphi_{ij}\in {\bb C}$. The following are equivalent:
\begin{itemize}
 \item[\rm (i)] $u_\varphi$ is bounded from $S_2$ to $S_p$.
\item[\rm (ii)] $\varphi$ admits a decomposition as
$
 \varphi = \psi+\chi$ {with}
$\sum\limits_i \sup\limits_j|\psi_{ij}|^r<\infty$ and $\sum\limits_j \sup\limits_i|\chi_{ij}|^r<\infty$.
\item[\rm (iii)] There is a sequence $f_i\ge 0$ with $\sum f_i<\infty$ such that 
$
 |\varphi_{ij}|\le f^{1/r}_i +f^{1/r}_j.
$
\end{itemize}
\end{thm}

\begin{proof}{(sketch)}  
(ii) $\Leftrightarrow$ (iii) is elementary, and (ii) $\Rightarrow$ (i) is easy. The main point is
(i) $\Rightarrow$ (ii). To prove this,
the scheme is the same as in \S \ref{sec3}. We again use extrapolation starting from the knowledge that Theorem \ref{thm4.1} holds when $p=q$ for some $q$ with $1\le q<2$. Let us fix $p$ with $0<p<1$. For any $q$ with $p\le q\le 2$, we denote
\[
 C'_q(\varphi) = \inf\{\|u_y\colon \ S_2\to S_q\|\}
\]
where the infimum runs over all $y=(y_{ij})$ for which there is $f_i\ge 0$ with $\sum f_i\le 1$ such that $\varphi_{ij} = (f^{\frac1p-\frac1q}_i y_{ij} + y_{ij}f^{\frac1p-\frac1q}_j)/2$. We also denote
\[
 ]\varphi[_p = \inf\{\|\psi\|_{\ell_r(\ell_\infty)} + \|^t\chi\|_{\ell_r(\ell_\infty)}\}
\]
where the infimum runs over all decompositions $\varphi = \psi+\chi$.

Note that $C'_p(\varphi) = \|u_\varphi\colon \ S_2\to S_p\|$. Let $1\le q<2$ and $\frac1q= \frac{1-\theta}p + \frac\theta2$. We have then by the same arguments as in \S \ref{sec1}:\ms 

\n {\bf Step 1$'$:}\ $]\varphi[_p \le C'C'_q(\lambda)$.

\n {\bf Step 2$'$:}\ $C'_2(\varphi) \le C''\ ]\varphi[_p$.

\n {\bf Step 3$'$:}\ $C'_q(\varphi) \le C'''C'_p(\varphi)^{1-\theta} C'_2(\varphi)^\theta$.\ms 

Note that obviously $\|u_y\colon \ S_2\to S_2\| = \sup|y_{ij}|$ so that we have again equivalence in Step 2$'$. To verify Step 3$'$ we argue exactly as for Theorem \ref{thm3.1}. 
\end{proof}
 
\begin{cor}\label{thm4.1+} Let $0<p\le 2\le q\le \infty$. Let $\frac 1r=\frac 1p -\frac 1q$.
With this value of $r$, the properties {\rm (i)} and {\rm (ii)} in the preceding Theorem 
  are equivalent to:
\begin{itemize}
 \item[\rm (i)'] $u_\varphi$ is bounded from $S_q$ to $S_p$.
\end{itemize}
\end{cor}
\begin{proof} Assume (i)'. Since $S_p$ has cotype 2 (\cite{X2}), $u_\varphi$ factors through a Hilbert space by \cite{P1}. By an elementary averaging argument (see e.g. \cite{PS}), the factorization can be achieved using only Schur multipliers. Thus we must have 
$\varphi=\varphi_1\varphi_2$ with
$\varphi_1$ (resp. $\varphi_2$) bounded
from $S_2$ to $S_p$ (resp. $S_q$ to $S_2$).
If we now apply Theorem \ref{thm4.1}  (resp. the results
of \cite{PS,X})  to $\varphi_1$ (resp. $\varphi_2$),
and use an arithmetic/geometric type  inequality of the form
$f^{\frac1p-\frac12} g^{\frac12-\frac1q}\le c(f^{\frac1p-\frac1q} +g^{\frac1p-\frac1q})$ for all $f,g\ge 0$,
we obtain   (iii).  The other implications are easy.
\end{proof}
\begin{prb} Characterize the bounded Schur multipliers   from $S_q$ to $S_p$ when $p<q<2$ or when $2<p<q\le \infty$.
\end{prb}
\n Some useful information on this problem can be derived from \cite{JP}. The difficulty      is due to the fact that, except when $q=1,2,\infty$, we have no characterization of the bounded Schur multipliers   on $S_q$.

\begin{rk} By general results, actually Theorem \ref{thm4.1} implies Theorem \ref{thm3.1}. Indeed, the same idea as in \cite{LPP} can be used to see this.
Moreover, as pointed out by Q. Xu, the converse implication is also easy: just observe that,
by Theorem \ref{thm3.1}, any Schur multiplier
bounded  from $S_2$ to $S_p$ must
be a bounded ''multiplier" from $\ell_2(\bb N\times \bb N)$ to $\ell_p(\ell_2 )+  {}^t \ell_p(\ell_2 ) $. Then
a well known variant of Maurey's classical factorization
yields (ii) or (iii) in Theorem \ref{thm4.1}.
 
\end{rk}

Although the recent paper \cite{JP} established several
important factorization theorems for
maps between non-commutative $L_p$-spaces,
there seems to be some extra difficulty to extend the Maurey factorization theorem  when $0<p<1$. The next result points to the obstacle. To avoid technicalities we again restrict to the finite dimensional case, so we assume $(M,\tau)$ as before but with $M$ finite dimensional. For any $\vp>0$, we denote
\[
 {\cl D}_\vp = \{f\in {\cl D}\mid f\ge \vp1\}.
\]
For any $x$ in $M$, we let
\[
 T(x)y = xy+yx.
\]
Note that if $x>0$ then $T(x)$ is an isomorphism on $M$ so that $T(x)^{-1}$ makes sense. \\

Let $B$ be any Banach space. Given a linear map $u\colon\ B\to L_p(\tau)$, we denote by $M_p(u)$ the smallest constant $C$ such that for any finite sequence $(x_j)$ in $B$
\[
 |||(ux_j)|||_p\le C\left(\sum \|x_j\|^2\right)^{1/2}.
\]
We denote by ${\cl M}_p(u)$ the smallest constant $C$ such that there is $\vp>0$ and a probability $\lambda$ on ${\cl D}_\vp$ such that
\begin{equation}\label{eq4.1}
 \forall x\in B\qquad\quad \int\|T(f^{\frac1r})^{-1} ux\|^2_2 d\lambda(f) \le C^2\|x\|^2
\end{equation}
where (as before) $\frac1r = \frac1p-\frac12$.

We have then

\begin{pro}\label{pro4.4}
 There is a constant $\beta>0$ such that for any $u$ as above we have
\[
\frac1\beta {\cl M}_p(u) \le M_p(u) \le {\cl M}_p(u).
\]
\end{pro}

\begin{proof}
The main point is to observe that if $\vp$ is chosen small enough (compared to $\dim(M)$) we have for any finite sequence $y=(y_j)$ in $L_p(\tau)$
\begin{equation}\label{eq4.2}
 |||(y_j)|||_p \le \inf_{f\in {\cl D}_\vp} \left(\sum\|T(f^{\frac1r})^{-1} y_j\|^2_2\right)^{1/2} \le \beta|||(y_j)|||_p
\end{equation}
where $\beta$ is a fixed constant, independent of the dimension of $M$. 

Then $M_p(u) \le {\cl M}_p(u)$ follows immediately. To prove the converse, assume $M_p(u)\le 1$. Then by \eqref{eq4.2} we have
\[
\inf_{y\in {\cl D}_\vp} \sum\|T(f^{\frac1r})^{-1} u(x_j)\|^2_2 \le \beta^2 \sum\|x_j\|^2.
\]
By a well known Hahn--Banach type argument (see e.g.\ Exercise 2.2.1 in \cite{P2}), there is a net $(\lambda_i)$ on ${\cl D}_\vp$ such that
\begin{equation}
\lim_i \int \|T(f^{\frac1r})^{-1} u(x)\|^2_2\ d\lambda_i(f) \le \beta^2\|x\|^2.\tag*{$\forall x\in B$}
\end{equation}
We may as well assume that the net corresponds to an ultrafilter. Setting $\lambda = \lim\lambda_i$, we obtain \eqref{eq4.1} and hence ${\cl M}_p(u)\le \beta$.
\end{proof}

\begin{rem}\label{rem4.5}
Now assume $1\le p < 2$. Note then that $\frac1r = \frac1p-\frac12$ satisfies $-1\le -\frac2r = 1-\frac2p < 0$. Therefore the function $t\to t^{-\frac2r}$ is \emph{operator convex} (see e.g. \cite[p. 123]{Bh}.
 Using this and assuming   ${\cl M}_p(u)\le 1$, we claim that  there is, for some $\vp>0$, a density $F$ in ${\cl D}_\vp$ such that
\begin{equation}\label{eq4.3}
\|T(F^{1/r})^{-1}ux\|_2 \le \beta'\|x\|\qquad \forall x\in H.
\end{equation}
Indeed we first observe that  
\begin{align}\label{eq4.4}
 \|T(f^{1/r})^{-1}y\|^2_2 &\simeq \|T(f)^{-\frac1r}y\|^2_2\\
&= \langle T(f)^{-\frac2r}y,y\rangle\notag
\end{align}
where $\simeq$ means that the (squared) norms   are equivalent
with equivalence constants depending only on $r$, and hence if we set
\[
 F = \int f\ d\lambda(f)
\]
we deduce from \eqref{eq4.1} that, for some constant $c$, we have
\[
 \langle T(F)^{-2/r} ux,ux\rangle \le c\|x\|^2
\]
and hence using \eqref{eq4.4} again we obtain \eqref{eq4.3}.  
\end{rem}
Note that
if $B$ is Hilbertian any bounded linear $u$ from $B$ to 
$L_p(\tau)$ satisfies the factorization of the form  \eqref{eq4.3}  if $1\le p\le 2$. This follows immediately by duality from
either \cite{LP2} or  \cite{LPX}.

 However,  what happens for $0<p<1$ is unclear :\ Can we still get rid of $\lambda$ as in the preceding remark?

\section{A non-commutative Kahane inequality}\label{sec5}

In vector-valued probability theory, the following inequalities due to Kahane (see \cite{Ka}) play an important role. For any $0<p<q<\infty$, there is a constant $K(p,q)$ such that for any Banach space $X$ and any finite sequence $(x_k)$ of elements of $X$ we have
\begin{equation}\label{eq5.1}
 \left\|\sum r_kx_k\right\|_{L_q(X)} \le K(p,q) \left\|\sum r_kx_k\right\|_{L_p(X)}
\end{equation}
where $(r_k)$ denotes as before the Rademacher functions.

As observed by C.~Borell (see \cite{Bo}) Kahane's result can be deduced from the hypercontractive inequality for the semi-group $T(t)$ defined on $L_2([0,1])$ by $T(t) \prod\limits_{k\in A} r_k = e^{-t|A|} \prod\limits_{k\in A} r_k$ for any finite set $A\subset {\bb N}$. The hypercontractivity says that if $1<p<q<\infty$ and if $e^{-2t}\le (p-1)(q-1)^{-1}$ then $\|T(t)\colon \ L_p\to L_q\|=1$. Since $T(t)\ge 0$ for all $t\ge 0$, this implies that for any Banach space $X$, we also have
\[
 \|T(t)\colon \ L_p(X) \to L_q(X)\|=1.
\]
In particular, if $S = \sum r_kx_k$ then $T(t)S = e^{-t}S$ and hence we find
\[
 \|S\|_{L_q(X)} \le (q-1)^{1/2} (p-1)^{-1/2}\|S\|_{L_p(X)},
\]
which yields \eqref{eq5.1} for $p>1$ (and the case  $0<p\le 1$ can be easily deduced from this using H\"older's inequality).

The goal of this section is to remark that this approach is valid mutatis mutandis in the ``anti-symmetric'' or Fermionic setting considered in \cite{CL}. Let $(M,\tau)$ be a von~Neumann algebra equipped with a faithful normal trace $\tau$ such that $\tau(1)=1$. Let $\{Q_k\mid k\ge 0\}$ be a spin system in $M$. By this we mean that $Q_k$ are self-adjoint unitary operators such that
\begin{equation}
 Q_kQ_\ell = - Q_\ell Q_k.\tag*{$\forall k\ne \ell$}
\end{equation}
For any finite set $A\subset {\bb N}$, ordered so that $A = \{k_1,\ldots, k_m\}$ with $ k_1 < k_2 <\cdots< k_m $, we set 
\[
 Q_A = Q_{k_1} Q_{k_2}\ldots Q_{k_m},
\]
with the convention
\[
 Q_\phi = 1.
\]
We will assume that $M$ is generated by $\{Q_k\}$. In that case, $M$ is the so-called hyperfinite factor of type $II_1$, i.e. the non-commutative analogue of
the Lebesgue interval $[0,1]$. 

Let $V(t)\colon \ L_2(\tau) \to L_2(\tau)$ be the semi-group defined for all $A\subset {\bb N}$ $(|A|<\infty)$ by
\[
 V(t)Q_A = e^{-t|A|}Q_A.
\]
Carlen and Lieb \cite{CL} observed that the semi-group $V(t)$ is completely positive (see \cite[(4.2) p. 36]{CL}) and proved that if $e^{-2t}\le (p-1)(q-1)^{-1}$
\[
 \|V(t)\colon\ L_p(\tau) \to L_q(\tau)\| = 1.
\]
We take the occasion of this paper to point out that the Kahane inequality remains valid in this setting provided one works with the ``vector-valued non-commutative $L_p$-spaces'' $L_p(\tau;E)$ introduced in \cite{P4}. Here $E$ is an operator space, i.e.\ $E\subset B(H)$ for some Hilbert space $H$, and $L_p(\tau;E)$ is defined as the completion if $L_p(\tau)\otimes E$ for the norm denoted by $\|\cdot\|_{L_p(\tau;E)}$ defined as follows.

For any $f$ in the algebraic tensor product  $L_p(\tau)\otimes E$
\begin{equation}\label{eq5.1+}
 \|f\|_{L_p(\tau;E)} = \inf\{\|a\|_{L_{2p}(\tau)} \|b\|_{L_{2p}(\tau)}\}
\end{equation}
where the infimum runs over all possible factorizations of $f$ of the form
\begin{equation}\label{eq5.2}
 f = a\cdot g\cdot b
\end{equation}
with $g\in M\otimes E$ such that
\[
 \|g\|_{M\otimes_{\min}E}\le 1.
\]
In \eqref{eq5.2}, the map $(a,g,b) \to a\cdot g\cdot b$ is obtained by linear extension from 
\[
 (a,(m\otimes e),b) \to amb\otimes e.
\]
Our observation boils down to the following.

\begin{lem}\label{lem5.1}
 If $T\colon \ L_p(\tau)\to L_q(\tau)$ $(1\le p, q\le\infty)$ is completely positive and bounded, then for any operator space $E\ne \{0\}$ the operator $T\otimes id_E$ extends to a bounded operator from $L_p(\tau;E)$ 
 to $L_q(\tau;E)$ such that
\[
 \|T\otimes id_E\colon \ L_p(\tau;E) \to L_q(\tau;E)\|  = \|T\colon\ L_p(\tau)\to L_q(\tau)\|.
\]
\end{lem}

\begin{proof}
 By density, it suffices to prove this when $M$ is generated by a finite subset $\{Q_0, Q_1,\ldots, Q_n\}$, say with even cardinality (i.e.\ $n$ odd). Then $M = M_{2^k}$ with $k=(n+1)/2$ and it is well known (see e.g. \cite{Pau}) that any c.p.\ map $T\colon \ M\to M$ is of the form
\begin{equation}\label{eq5.3}
 T(x) = \sum a^*_j x a_j
\end{equation}
for some finite set $a_j$ in $M$.

Now assume $f\in L_p(\tau)\otimes E$ with $\|f\|_{L_p(\tau;E)}<1$. We can write $f = a^*\cdot g\cdot b$ with $\|a\|_{2p}, \|b\|_{2p}<1$ and $\|g\|_{M\otimes_{\min} E}<1$. Assume $\|T\colon \ L_p(\tau)\to L_q(\tau)\| = 1$. Let $\alpha = \left(\sum a^*_ja^*aa_j\right)^{1/2}$ and $\beta = \left(\sum a^*_jb^*ba_j\right)^{1/2}$. Since $\alpha^2 = T(a^*a)$ and $\beta^2=T(b^*b)$ we have $\|\alpha\|_{2q}<1$ and $\|\beta\|_{2q}<1$. Fix $\vp>0$. We have
\begin{align*}
 aa_j &= \alpha_j\  (\vp 1+\alpha^2)^{1/2}\\
ba_j &= \beta_j\  (\vp1+\beta^2)^{1/2}
\end{align*}
where $\alpha_j =a a_j(\vp1+\alpha^2)^{-1/2}$ and $\beta_j = ba_j(\vp1+\beta^2)^{-1/2}$ satisfy 
\[
 \sum \alpha^*_j \alpha_j = (\vp1+\alpha^2)^{-1/2} \alpha^2(\vp1+\alpha^2)^{1/2} \le 1
\]
and similarly $\sum \beta^*_j\beta_j \le 1$. This implies clearly (by the defining property of an operator space!)
\[
 \left\|\sum \alpha^*_j\cdot g\cdot \beta_j\right\|_{M\otimes_{\min} E} < 1.
\]
We have
\[
 (T\otimes id_E)(f) = (\vp1 + \alpha^2)^{1/2} \hat g(\vp 1+\beta^2)^{1/2}
\]
where $\hat g = \sum \alpha^*_j\cdot g\cdot \beta_j$, and hence we conclude by \eqref{eq5.1+}
\[
 \|(T\otimes id_E)(f)\|_{L_q(\tau;E)} \le \|(\vp1+\alpha^2)^{1/2}\|_{2q} \|(\vp1+\beta^2)^{1/2}\|_{2q} \le (\vp + \|\alpha^2\|_q)^{1/2} (\vp+\|\beta^2\|_q)^{1/2}\le 1+\vp,
\]
and since $\vp>0$ is arbitrary, we obtain the announced result by homogeneity.
\end{proof}

\begin{rem}
Q.~Xu pointed out to me that Lemma \ref{lem5.1} remains valid in the nonhyperfinite case.One can check this using the following fact:\ consider $y$ in $L_p(\tau)\otimes M_n$, then $y\in B_{L_p(\tau;M_n)}$ iff there are $\lambda,\mu$ in $B_{L_p(\tau)}$ such that $\left(\begin{smallmatrix} \lambda&y\\ y^*&\mu\end{smallmatrix}\right) \ge 0$ where $\ge 0$ is meant in $L_p(\tau\times \text{tr})$ (see e.g.\ Exercise 11.5 in \cite{P2} for the result at the root of this fact). A similar statement is valid with $B(H)$ in place of $M_n$.
\end{rem}

\begin{thm}\label{thm5.2}
 Let $1<p<q<\infty$. Assume $e^{-2t} \le (p-1)(q-1)^{-1}$, then for any operator space $E$
\[
 \|V(t)\colon \ L_p(\tau;E) \to L_q(\tau;E)\|\le 1.
\]
Consequently, for any $1\le p<q<\infty$ there is a constant $K'(p,q)$ such that for any $E$ and any finite sequence $x_k$ in $E$ we have
\[
 \left\|\sum Q_k\otimes x_k\right\|_{L_q(\tau;E)} \le K'(p,q) \left\|\sum Q_k\otimes x_k\right\|_{L_p(\tau;E)}.
\]
\end{thm}

\begin{proof}
The first part follows from the preceding Lemma by \cite{CL}. Let $f = \sum Q_k\otimes x_k$. In particular, if $1<p<q<\infty$ we have
\begin{equation}\label{eq6.11}
 \|f\|_{L_q(\tau;E)} \le (q-1)^{1/2} (p-1)^{-1/2} \|f\|_{L_p(\tau;E)}.
\end{equation}
Let $0<\theta <1$ be defined by
\[
 \frac1p = \frac{1-\theta}1 + \frac\theta{q}.
\]
By \cite[p. 40]{P4} we have isometrically
\[
 L_p(\tau;E) = (L_1(\tau;E), L_q(\tau;E))_\theta
\]
and hence
\[
 \|f\|_{L_p(\tau;E)} \le \|f\|^{1-\theta}_{L_1(\tau;E)} \|f\|^\theta_{L_q(\tau;E)},
\]
which when combined with \eqref{eq6.11}  yields 
\[
 \|f\|_{L_q(\tau;E)} \le ((q-1)^{1/2}(p-1)^{-1/2})^{\frac1{1-\theta}} \|f\|_{L_1(\tau;E)}.\qquad\qed
\]
\renewcommand{\qed}{}\end{proof}

\begin{rk} Obviously  Theorem \ref{thm5.2} is also valid for   other   hypercontractive semi-groups, as the ones in \cite{Bi}.

\end{rk}

\section{Appendix}\label{sec10}

The main technical difficulty in our proof of Step 3 above is \eqref{eq1.6}.
We will first show how this follows from 
Theorem 1.1 in \cite{JP}. We will then also outline a direct more self-contained argument.

 Let $(M,\tau)$ be a generalized (possibly non-commutative) measure space, with associated space $L_p(\tau)$.
 Since it is easy to pass from the finite
 to the semifinite case, we assume $\tau$ finite. Consider a density $f>0$ in $M$ with $\tau(f) = 1$, with finite spectrum, i.e.   we assume that $f = \sum^N_1 f_jQ_j$ where   $0<f_1\le f_2\le...\le f_N$,  $1 = \sum^N_1 Q_j$ and $Q_j$ are 
 mutually orthogonal projections in $M$. We now introduce for any $x$ in $L_p(\tau)$ ($1\le p\le \infty$)
\begin{equation}\label{eq10.1}
\|x\|_{L_p(f)} = \|f^{\frac1p}x\|_{L_p(\tau)} + \|xf^{\frac1p}\|_{L_p(\tau)}.
\end{equation}
We will denote by $L_p(f)$ the space $L_p(\tau)$ equipped with the norm $\|\cdot\|_{L_p(f)}$. Then   \cite[Th. 1.1]{JP} implies in particular that for any $0<\theta<1$ and any $1<p<\infty$ we have
\begin{equation}\label{eq10.2}
 L_{p(\theta)}(f) \simeq (M,L_p(f))_\theta
\end{equation}
where $p(\theta)^{-1} = \frac{1-\theta}\infty + \frac\theta{p} = \theta/p$, and where $\simeq$ means that the norms on both sides are equivalent with equivalence constants depending only on $p$ and $\theta$. Note that by the triangle inequality and by Lemma \ref{lem1.6} (ii), we have
\begin{equation}\label{eq10.3}
\|f^{\frac1p}x + xf^{\frac1p}\|_{L_p(\tau)} \le \|x\|_{L_p(f)} \le 2t(p)\|f^{\frac1p}x + xf^{\frac1p}\|_{L_p(\tau)}.
\end{equation}
Let us denote
\[
 T(f)x = fx + xf.
\]
With this notation, the dual norms
\[
\|x\|_{L_p(f)^*} = \sup \{|\tau(xy)| \ \big| \  \|y\|_{L_p(f)}\le 1\}
\]
satisfy for any $x$ in $L_{p'}(\tau)$ the following dual version to \eqref{eq10.3}
\begin{equation}\label{eq10.4}
 (2t(p))^{-1} \|T(f^{\frac1p})^{-1}x\|_{L_{p'}(\tau)} \le \|x\|_{L_p(f)^*} \le \|T(f^{\frac1p})^{-1}x\|_{L_{p'}(\tau)}.
\end{equation}
Note that with our simplifying assumptions on $f$, $T(f)$ is an isomorphism on $L_p(\tau)$.

Here and in the sequel we will denote by $c_1,c_2,\ldots$ constants depending only on $p$ and $\theta$.\\
Recall (see e.g.\ \cite{BL}) that we have isometrically for any $0<\theta<1$
\[
 (M,L_p(f))^*_\theta = (L_1(\tau), L_p(f)^*)_\theta.
\]
Therefore \eqref{eq10.2} implies in particular that for any $x$ in $L_{p'}(\tau)$ 
\begin{equation}\label{eq10.5}
 \|x\|_{L_{p(\theta)}(f)^*} \le c_1\|x\|^{1-\theta}_{L_1(\tau)} \|x\|^\theta_{L_p(f)^*}.
\end{equation}
 Using \eqref{eq10.4}, \eqref{eq10.5} implies
\begin{equation}\label{eq10.5a}
\|T(f^{\frac\theta{p}})^{-1}(x)\|_{L_{p(\theta)'}(\tau)} \le c_2 \|x\|^{1-\theta}_{L_1(\tau)} \|T(f^{\frac1p})^{-1}x\|^\theta_{L_{p'}(\tau)}.
\end{equation}
  In Step 3 of the present paper, we used the special case $p=2$. If we denote $q=p(\theta)'$ we have $\frac1q = \frac{1-\theta}1  + \frac\theta2$ so that \eqref{eq10.5a} becomes
\begin{equation}\label{eq10.6}
 \|T(f^{\frac\theta2})^{-1}(x)\|_q \le c_2\|x\|^{1-\theta}_1 \|T(f^{\frac12})^{-1}x\|^\theta_2,
\end{equation}
and we obtain \eqref{eq1.6} for $p=1$. The case $1<p<2$ can   be derived by the same argument, but
this is anyway much easier because of the simultaneous boundedness on $L_p$ and $L_2$
of the triangular projection.
$\hfill\square$

For the convenience of the reader, we now give a direct argument, based on the same ideas as \cite{JP}. We want to show \eqref{eq10.6}. Note that it is equivalent to (change $x$ to $T(f^{\frac12}y$)) : \ for all $y$ in $M$
\begin{equation}\label{eq10.7}
 \|T(f^{\frac\theta2})^{-1} T(f^{\frac12})y\|_q \le c_4 \|T(f^{\frac12})y\|^{1-\theta}_1 \|y\|^\theta_2.
\end{equation}
By the triangle inequality and by Lemma \ref{lem1.6} (ii) we have
\[
\|T(f^{\frac\theta2})^{-1} T(f^{\frac12})y\|_q \le \|T(f^{\frac\theta2})^{-1} f^{\frac\theta2} f^{\frac{1-\theta}2}y\|_q + \|T(f^{\frac\theta2})^{-1} yf^{\frac{1-\theta}2} f^{\frac\theta2}\|_q \le t(q)(\|f^{\frac{1-\theta}2}y\|_q + \|yf^{\frac{1-\theta}2}\|_q).
\]
Therefore to show \eqref{eq10.6} (or \eqref{eq10.7}) it suffices to show
\begin{equation}\label{eq10.8}
 \|f^{\frac{1-\theta}2}y\|_q + \|yf^{\frac{1-\theta}2}\|_q \le c_6\|f^{\frac12}y + yf^{\frac12}\|^{1-\theta}_1 \|y\|^\theta_2.
\end{equation}
Recall that $f = \sum^N_1 f_jQ_j$. We denote
\[
 y^+ = \sum_{i\le j} Q_iyQ_j,\quad  y^- =\sum_{i>j} Q_iyQ_j.
\]
Note that $y^+$ (resp.\ $y^-$) is the upper (resp.\ lower) triangular part of $y$ (with respect to the decomposition $I=\sum Q_j$). We recall that,  whenever $1<q<\infty$, $y\mapsto y^+$ and $y\mapsto y^-$ are bounded linear maps on $L_q(\tau)$ with bounds independent of $N$, but this fails in case $q=1$ or $q=\infty$  (see \cite{Mat} and  \cite[\S 8]{PX} for references on this).

By the triangle inequality, since $y=y^++y^-$, to prove \eqref{eq10.8} it suffices to show $\forall y\in M$
\begin{equation}\label{eq10.9}
 \max\{\|f^{\frac{1-\theta}2}y^+\|_q, \|y^+f^{\frac{1-\theta}2}\|_q\} \le c_7\|f^{\frac12}y + yf^{\frac12}\|^{1-\theta}_1 \|y\|^\theta_2
\end{equation}
and similarly with $y^-$ in place of $y^+$. Let $L^-_p(\tau) =\{x\in L_p(\tau)\mid x^+=0\}$. Let $\Lambda_p = L_p(\tau)/L^-_p(\tau)$. Note that $x^++L^-_p(\tau) = x+L^-_p(\tau)$. We will denote abusively by $\|x^+\|_{\Lambda_p}$ the norm in $\Lambda_p$ of the equivalence class of $x^+$ modulo $L^-_p(\tau)$. Note that $\|x^+\|_{\Lambda_1} \le \|x\|_1$ for all $x$ in $L_1(\tau)$ and hence $\|f^{\frac12}y^+ + y^+ f^{\frac12}\|_{\Lambda_1} \le \|f^{\frac12}y+yf^{\frac12}\|_1$ for all $y$ in $L_2(\tau)$. Moreover we have $\|y^+\|_{\Lambda_2} = \|y^+\|_2$. Therefore to show \eqref{eq10.9} it suffices to show
\begin{equation}\label{eq10.10}
 \|f^{\frac{1-\theta}2}y^+\|_q \le c_7\|f^{\frac12}y^+ + y^+f^{\frac12}\|^{1-\theta}_{\Lambda_1} \|y^+\|^\theta_{\Lambda_2}
\end{equation}
and similarly for $y^+f^{\frac{1-\theta}2}$.

We now observe that, by Lemma \ref{lem1.6} (i), the maps
\[
T_1\colon \ x\mapsto \sum \frac{\lambda_i\wedge \lambda_j}{\lambda_i+\lambda_j} Q_ixQ_j\quad \text{and}\quad T_2\colon \ x\mapsto \sum \frac{\lambda_i\vee \lambda_j}{\lambda_i+\lambda_j} Q_ixx_j
\]
have norm $\le 3/2$ on $L_q(\tau)$ for all $1\le q\le \infty$, in particular on $L_1(\tau)$. Since these maps preserve $L^-_1(\tau)$, the ``same'' maps are contractive on $\Lambda_1$. Applying this with $\lambda_i = f^{1/2}_i$ and assuming as before that $f_1\le \ldots\le f_N$, we have $f_i\wedge f_j = f_i$ and $f_i\vee f_j = f_j$ for all $i\le j$ and hence $T_1(f^{\frac12}y^+ +y^+f^{\frac12})=f^{1/2}y^+$ and $T_2(f^{\frac12}y^+ + y^+f^{\frac12}) = y^+f^{\frac12}$. This gives us 
\[
 \max\{\|f^{\frac12}y^+\|_{\Lambda_1},\quad \|y^+f^{\frac12}\|_{\Lambda_1}\} \le (3/2)\|f^{\frac12}y^+ + y^+ f^{\frac12}\|_{\Lambda_1}.
\]
Thus to show \eqref{eq10.10} it suffices to show 
\[
 \|f^{\frac{1-\theta}2}y^+\|_q \le c_7\|f^{\frac12}y^+\|^{1-\theta}_{\Lambda_1} \|y^+\|^\theta_{\Lambda_2},
\]
and similarly for $y^+f^{\frac{1-\theta}2}$. Now by  \cite[Th.~4.5]{P3} and by duality we have $(\Lambda_1,\Lambda_2)_\theta\simeq \Lambda_q$ with equivalent norms (and equivalence constants independent of $N$). Using the analytic function $z\mapsto f^{\frac{z}2}$ and a by now routine application of the 3 line lemma (this is essentially the ``Stein interpolation principle'') this gives us (recall   $\|y^+\|_{\Lambda_2} = \|y^+\|_2$)
\[
\|f^{\frac{1-\theta}2}y^+\|_{\Lambda_q} \le c_8\|f^{\frac12}y^+\|^{1-\theta}_{\Lambda_1} \|y^+\|^\theta_2.
\]
But now since the ``triangular projection'' $y\mapsto y^+$ is bounded on $L_q(\tau)$ when $1<q<\infty$ (and since $(f^{\frac{1-\theta}2}y)^+ = f^{\frac{1-\theta}2}y^+$) we obtain finally
\[
 \|f^{\frac{1-\theta}2}y^+\|_q\le c_9\|f^{\frac12}y^+\|^{1-\theta}_{\Lambda_1} \|y^+\|^\theta_2.
\]
By the preceding successive reductions, this completes the proof of \eqref{eq10.6} and hence also of \eqref{eq1.6} for $p=1$.$\hfill\square$

\vspace{1cm}

\n\textbf{Acknowledgement.} I am  very grateful to
  Quanhua Xu  for   many stimulating suggestions and improvements. I also thank the referee
  for his/her very careful reading and the resulting corrections.

  \end{document}